\documentclass[a4paper,11pt]{article}

\usepackage{graphicx}
\usepackage{amsmath}
\usepackage{amsfonts}
\usepackage{amsthm}
\usepackage{amssymb}
\usepackage{geometry}
\usepackage{color}
\usepackage{accents}
\usepackage{epic}

\usepackage{mathtools}
\usepackage{framed}
\usepackage{graphicx, caption, subcaption}
\usepackage{hyperref}
\usepackage{dsfont}
\usepackage{float}

\usepackage{algorithm2e}

\newcommand{\R}{\mathbb{R}}

\newcommand{\dx}{\operatorname{\,d\hspace{-0.4ex}}x}
\newcommand{\dt}{\operatorname{\,d\hspace{-0.4ex}}t}
\newcommand{\ds}{\operatorname{\,d\hspace{-0.4ex}}s}

\newcommand{\eps}{\varepsilon}

\newcommand{\quotes}[1]{``#1''}

\definecolor{dark-green}{rgb}{0.0,0.4,0.0}
\definecolor{pantone_369}{rgb}{0.4784, 0.7098, 0.0863}
\definecolor{lila}{rgb}{0.5,0.0,0.5}

\newtheorem{theorem}{Theorem}[section]
\newtheorem{corollary}[theorem]{Corollary}

\newtheorem{lemma}[theorem]{Lemma}
\newtheorem{proposition}[theorem]{Proposition}

\theoremstyle{definition}
\newtheorem{definition}[theorem]{Definition}
\newtheorem{remark}[theorem]{Remark}
\usepackage{framed}

\newenvironment{fshaded}{%
\MakeFramed {\FrameRestore}}%
{\endMakeFramed}

\newtheorem{cstep}{Step}

\begin{document}

\begin{center}
{\LARGE Sobolev gradient flow for the Gross-Pitaevskii eigenvalue problem: global convergence and computational efficiency
\renewcommand{\thefootnote}{\fnsymbol{footnote}}\setcounter{footnote}{0}
 \hspace{-3pt}\footnote{P. Henning acknowledges funding by the Swedish Research Council (grant 2016-03339) and the G\"oran Gustafsson foundation and D.~Peterseim acknowledges support by the German Research Foundation DFG in the Priority Program 1748 ``Reliable simulation techniques in solid mechanics'' (PE2143/2-2). Parts of this paper were written while the authors enjoyed the kind hospitality of the Fields Institute in Toronto.}}\\[2em]
\end{center}

\renewcommand{\thefootnote}{\fnsymbol{footnote}}
\renewcommand{\thefootnote}{\arabic{footnote}}

\begin{center}
{\large Patrick Henning\footnote[1]{Department of Mathematics, KTH Royal Institute of Technology, SE-100 44 Stockholm, Sweden.} and 
Daniel Peterseim\footnote[2]{Institut f\"ur Mathematik, Universit\"at Augsburg, Universit\"atsstr. 14, DE-86159 Augsburg, Germany}}\\[2em]
\end{center}


\begin{center}
{\large{\today}}
\end{center}

\begin{center}
\end{center}

\begin{abstract}
We propose a new normalized Sobolev gradient flow for the Gross-Pitaevskii eigenvalue problem based on an energy inner product that depends on time through the density of the flow itself. The gradient flow is well-defined and converges to an eigenfunction. For ground states we can quantify the convergence speed as exponentially fast where the rate depends on spectral gaps of a linearized operator. The forward Euler time discretization of the flow yields a numerical method which generalizes the inverse iteration for the nonlinear eigenvalue problem. For sufficiently small time steps, the method reduces the energy in every step and converges globally in $H^1$ to an eigenfunction. In particular, for any nonnegative starting value, the ground state is obtained. A series of numerical experiments demonstrates the computational efficiency of the method and its competitiveness with established discretizations arising from other gradient flows for this problem.
\end{abstract}

\section{Introduction}
A Bose-Einstein condensate (BECs) is an extreme state of matter formed by a dilute gas of bosons at ultra-cold temperatures, very close to absolute zero \cite{Bos24,DGP99,Ein24,PiS03}. In a BEC, individual particles (i.e. their wave packages) overlap, lose their identity, and form one single \quotes{super atom}. BECs allow to study macroscopic quantum phenomena such as superfluity (i.e. the frictionless flow of a fluid) on an observable scale. This is why BECs are a very relevant research area of modern quantum physics \cite{Alaeian_2017,ARV01,Fet09,Leg03,MAH99}. For a general mathematical description of BECs and corresponding analytical results we refer exemplarily to \cite{Aft06,BaC13b,BrS19,LSY01,PiS03}.

In this paper we consider stationary states of a BEC modeled by the Gross-Pitaevskii eigenvalue problem (GPE) in real-valued variables. In non-dimensional form, the GPE seeks $L^2$-normalized eigenfunctions $z^{\ast} \in H^1(\R^d)$ and corresponding eigenvalues $\lambda^{\ast} \in \R$ such that 
\begin{align*}
 - \triangle z^{\ast} + V z^{\ast} + \beta |z^{\ast}|^2 z^{\ast} = \lambda z^{\ast}.
\end{align*}
In the context of Bose-Einstein condensates, a solution $z^{\ast}$ represents a stationary quantum state of the condensate, $|z^{\ast}|^2$ is the corresponding density and $\lambda^{\ast}$ the so-called chemical potential. The function $V$ represents an external confining potential and the parameter $\beta$ depends on physical properties of the particles that form the BEC. Its sign determines the type of particle interactions. In this paper, we shall only consider the defocusing GPE, which covers the regime $\beta\ge0$, resembling repulsive particle interactions. The normalization constraint $\int_{\R^d} |z^{\ast}|^2 \dx=1$ is such that the total mass of the condensate equals the number of constituting particles (with probability $1$).

The numerical solution of the stationary GPE has been studied extensively in recent years; see e.g.  \cite{AfD01,ALT17,BCL06,BaD04,BJM03,BaS08,BaT03,COR09,CCH18,CCM10,CaL00,CGZ10,CHJ08,DaK10,DaP17,DiC07,FaT18,GaP01,HMP14b,JKM14,KaE10,RSB14,RSS09,XiX16} and the references therein.
Typically, the problem is rephrased  in terms of the energy functional 
$$
E(v) := \tfrac{1}{2}\int_{\R^d} |\nabla v|^2  + V \hspace{2pt} |v|^2  + \tfrac{\beta}{2} |v|^4\dx,
$$
where one is interested in finding the critical points of $E$ under the normalization constraint $\| v\|_{L^2(\R^d)}=1$. The unique global minimizer (the state of minimal energy) is called the ground state, whereas all other critical points are called excited states. The identification of critical points of $E$ can be accomplished by the construction of appropriate gradient flows of the form
\begin{align}
\label{intro-grad-flow}
z^{\prime}(t) = - P_{z,X}( \nabla_X E(\,z(t)\,) ),
\end{align}
where $\nabla_X E$ is the Sobolev gradient of the energy functional $E$ with respect to some inner product $(\cdot,\cdot)_X$ and where $P_{z,X}$ is the projection onto the tangent space associated with the normalization constraint. Depending on the choice of $(\cdot,\cdot)_X$ and the numerical time integration of the arising gradient flow, several numerical methods arise (cf. \cite{DaK10}). Presumably, the most popular method in the context of the GPE is the {\it Discrete Normalized Gradient Flow} (DNGF) \cite{BaD04} which is based on the choice $(\cdot,\cdot)_X=(\cdot,\cdot)_{L^2}$ of the $L^2$-inner product and a backward Euler-type time discretization with explicit treatment of the nonlinear term. Other approaches combine a forward Euler discretization with the choice $(\cdot,\cdot)_X=(\cdot,\cdot)_{H^1}$ \cite{KaE10,RSS09} or the choice $(\cdot,\cdot)_X=(\nabla \cdot ,\nabla \cdot)_{L^2} + (V \cdot , \cdot )_{L^2}$ \cite{DaK10}. These examples and their discrete version are briefly discussed in Section~\ref{sec:continuous:grad:flow}. For further variants, we refer to \cite{AfD01,ALT17,BaS08,DaP17,GaP01,RSB14}.

Although the aforementioned schemes for the GPE are empirically successful, their numerical analysis lacks a proof of global convergence to a critical point of $E$ and any quantification of convergence rates. There is not even a proof of monotonic energy dissipation of the iteration in analogy to the continuous gradient flow \eqref{intro-grad-flow}. The only result that comes close is for DNGF (based on the $L^2$-gradient) \cite{BaD04}. In the absence of any spatial discretization, the reduction of a modified energy is shown which deviates from the exact energy by a term of the form $\tfrac{\beta}{4}\| v \|_{L^4}^4$. Since this result exploits elliptic regularity theory, its generalization to a fully discrete setting involving e.g. a finite element discretization is not straight forward.

In Section~\ref{sec:discrete:grad:flow} of this paper, we present a new choice for the Sobolev gradient, where the inner product $(\cdot,\cdot)_X$ is not fixed, but evolves with time. It is selected in such a way that the Sobolev gradient equals the identity, thus leading to an optimal preconditioning of the flow. We show that the arising continuous gradient flow of the form \eqref{intro-grad-flow} is well-posed. Thanks to the optimal preconditioning, the problem can be discretized by the forward Euler scheme (cf. Section~\ref{sec:discrete:grad:flow}). The time-discrete method reduces the (correct) energy monotonically and converges globally in $H^1$ to a critical point of $E$ for sufficiently small time steps. These unique results remain valid even after Galerkin discretization in space. Furthermore, in Section~\ref{sec:positive:eigenstates} we prove that, for any non-negative initial value $z_0$, the method must necessarily converge to a strictly positive eigenfunction of the GPE. Since there exist no positive excited state, the method is guaranteed to converge to the ground state whenever $z_0\ge0$.

Exponential convergence of the new discrete Gradient flow with respect to the number of iterations (i.e. reduction of the error by a fixed factor in each step) remains open but is observed in numerical experiments. It is worth mentioning that, for a particular choice of the time step, the method recovers the inverse iteration for the nonlinear eigenvalue problem. 
Moreover, for this very time step, the method is equivalent to DNGF which indicates its competitiveness with the established approaches for the GPE. In some scenarios we even observe superior performance (see Section~\ref{sec:num:experiments}). This is particularly true when the time step is chosen adaptively by some standard line search strategy which appears to be cost neutral.

\section{Model problem and established gradient flows}
\label{sec:preliminaries}

We shall introduce the precise setup of the model problem of this paper and briefly recall the projected $L^2$- and $H^1$-Sobolev gradient flows at hand. Note that all functions and functionals considered in this paper are real-valued.

\subsection{Gross-Pitaevskii eigenvalue problem}

Since confinement potentials $V$ cause a localization of stationary states, it is common to consider the GPE on a bounded domain $\Omega  \subset \R^{d}$, for $d=1,2,3$, together with a homogenous Dirichlet boundary condition. In addition to the boundedness, we shall also assume that $\Omega$ is either a convex Lipschitz domain or a domain with a smooth boundary. The latter assumption is natural in this context and prevents singular behavior of stationary states at the artificial boundary. 
We also assume that the nonlinearity is defocusing, i.e., $\beta\ge0$, and that the potential is bounded almost everywhere, i.e. $V\in L^{\infty}(\Omega)$. Without loss of generality, we assume that $V\ge 1$ a.e., as a constant shift of $V$ would not affect the eigenfunctions but only shift the spectrum accordingly. Note that this assumption implies that all eigenvalues satisfy $\lambda^{\ast}>1$. We note that we only use $V\ge 1$ instead of $V\ge 0$ to avoid a repeated usage of the Poincar\'e inequality in our estimates.

We define the non-negative energy for a function $v\in H^1_0(\Omega)$ by
$$
E(v) := \tfrac{1}{2}\int_{\Omega}  |\nabla v|^2 +  V \hspace{2pt} |v|^2 + \tfrac{\beta}{2} |v|^4\dx.
$$
The energy functional is strictly convex and Fr\'echet differentiable, where the first Fr\'echet derivative is given by
\begin{align*}
\langle E^{\prime}(v) , w \rangle = \int_{\Omega} \nabla v \cdot \nabla w + V \hspace{2pt} v \hspace{2pt} w + \beta \hspace{2pt}|v|^2 v \hspace{2pt} w\dx.
\end{align*}
Here, $\langle \cdot,\cdot\rangle$ denotes the dual pairing between $H^{-1}(\Omega)$ and $H^{1}_0(\Omega)$.
The Gross-Pitaevskii eigenvalue problem (GPE) seeks the critical points $z^{\ast}\in H^1_0(\Omega)$ of the energy functional $E$ subject to the constraint $\| z^{\ast} \|_{L^2(\Omega)}=1$. A function $z^{\ast}$ is a critical point if there is a $\lambda^{\ast} \in \R$ such that
\begin{align}\label{weak-problem}
\langle E^{\prime}(z^{\ast}) , v \rangle = \lambda^{\ast} ( z^{\ast} , v )_{L^2(\Omega)} \qquad \mbox{for all } v\in H^1_0(\Omega).
\end{align}
For an $L^2$-normalized eigenfunction $z^{\ast}$, the energy is related to the corresponding eigenvalue $\lambda^{\ast}$ through the equality
\begin{align}
\label{groundstate-energy-formula} \lambda^{\ast} = 2 \hspace{2pt} E(z^{\ast}) + \tfrac{\beta}{2} \| z^{\ast} \|_{L^4(\Omega)}^4.
\end{align}
Classical Ljusternik-Schnirelman theory  (cf. \cite{Zei85}) for even, positive, convex functionals guarantees that problem \eqref{weak-problem} has infinitely many eigenvalues $0<\lambda_1^{\ast}<\lambda_2^{\ast}\le\lambda_3^{\ast} \le \cdots < \infty$. Of particular interest is the ground state of $E$ (the global minimizer) with ground state eigenvalue $\lambda_{\mbox{\rm\tiny GS}}=\lambda_1^{\ast}$.
The following result can be e.g. found in \cite{CCM10}.
\begin{proposition}\label{basic-existence-ground-state}
Under the general assumptions of this paper there exists a ground state $z_{\mbox{\rm\tiny GS}} \in H^1_0(\Omega)$ with $\| z_{\mbox{\rm\tiny GS}} \|_{L^2(\Omega)}=1$ such that  
\begin{align*}
E( z_{\mbox{\rm\tiny GS}} ) = \inf \{ E(z) \hspace{3pt} | \hspace{3pt} z\in H^1_0(\Omega) \mbox{ \rm and } \| z \|_{L^2(\Omega)} =1 \}.
\end{align*}
The (normalized) ground state is unique up to its sign, it is H\"older-continuous on $\overline{\Omega}$ and it satisfies $|z_{\mbox{\rm\tiny GS}}|>0$ in $\Omega$. Furthermore, the Lagrange multiplier $\lambda_{\mbox{\rm\tiny GS}}$ given by \eqref{groundstate-energy-formula}, is the smallest eigenvalue of the GPE \eqref{weak-problem} with corresponding eigenfunction $z_{\mbox{\rm\tiny GS}}$. This ground state eigenvalue $\lambda_{\mbox{\rm\tiny GS}}$ is simple.
\end{proposition}
We stress the nontrivial observation that a normalized eigenfunction to the smallest eigenvalue of the GPE is always a global minimizer of the energy functional $E$. We are not aware of any result that ensures that the ordering of the eigenvalues $\lambda_n^{\ast}$ by size still corresponds with the ordering of the energies $E(z_n^{\ast})$ by size for $n>1$.

Other than the ground state, excited states are not unique (up to sign) in general. E.g., on a circular domain with an isotropic quadratic potential, the eigenvalues that correspond to excited states can even have an infinite multiplicity due to rotational invariance of $E$. 

%

\subsection{Projected Sobolev gradient flows}
We shall briefly recall the basic concept of projected gradient flows. For a detailed introduction to the topic in the context of the Gross-Pitaevskii equation, we refer to \cite{KaE10}. 

We consider the energy functional $E$ along with a Hilbert space $H^1_0(\Omega) \subset X \subset  L^2(\Omega)$ with inner product $(\cdot,\cdot)_X$ as the energy dissipation mechanism. Various choices for $X$ are possible and lead to different gradient flows. With this, let $\nabla_X E(z)$ denote the Riesz-representative of $E^{\prime}(z)$ in the space $X$, i.e., $\nabla_X E(z) \in X$ satisfies
\begin{align}
\label{def-sobolev-gradient}
( \nabla_X E(z)  , v )_{X} = \langle E^{\prime}(z) , v \rangle \qquad \mbox{for all } v \in X.
\end{align}
The operator $\nabla_X E : H^1_0(\Omega) \rightarrow X$ is called the {\it Sobolev gradient of $E$ with respect to $X$}. For the sake of mass conservation along the flow, we define the tangent space of the constraint $\| z \|_{L^2(\Omega)}^2=1$ in $X$ by
$$
T_{z,X} := \{ v \in X | \hspace{2pt} (v,z)_{L^2(\Omega)}=0 \}.
$$
Note that $T_{z,X}$ is the null space of the Fr\'echet derivative of the functional $v\mapsto\| v \|_{L^2(\Omega)}^2$ on $X$ evaluated at $z$. If $z^{\prime}=P(z)$ for some operator $P: H^1_0(\Omega) \rightarrow T_{z,X}$, then we have $z^{\prime} \in X \subset L^2(\Omega)$ and, hence,
$$
\frac{1}{2} \frac{\mbox{d}}{\mbox{d}t} \| z(t) \|_{L^2(\Omega)}^2 = (z^{\prime}(t) , z(t))_{L^2(\Omega)} = (P(\hspace{1pt}z(t)\hspace{1pt}) , z(t))_{L^2(\Omega)} =0,
$$
i.e., we have mass conservation with $\| z(t) \|_{L^2(\Omega)}^2=\| z(0) \|_{L^2(\Omega)}^2$ for all $t\ge0$.
This motivates to seek the best approximation of the Sobolev gradient $\nabla_X E(z)$ in the tangent space $T_{z,X}$. The $X$-orthogonal projection $P_{z,X}(v) \in T_{z,X}$ of $v$ onto $T_{z,X}$ is given by
\begin{align*}
( P_{z,X}(v) , \psi )_X = ( v  , \psi )_X \qquad \mbox{for all } \psi \in T_{z,X},
\end{align*}
and be expressed in terms of the Riesz-representative $R_X(z) \in X$ of $z$ in $X$ by 
$$
 P_{z,X}(v) = v -  \frac{(z,v)_{L^2(\Omega)}}{(z,R_X(z))_{L^2(\Omega)}} R_X(z).
$$
Given some sufficiently smooth initial value $z_0$, the projected Sobolev gradient flow is then characterized by  
\begin{align*}
z^{\prime}(t)  = - (P_{z(t),X} \circ \nabla_X E)(\hspace{1pt}z(t)\hspace{1pt}) \qquad \mbox{for } t \ge 0
\qquad \mbox{and } \qquad z(0)=z_0.
\end{align*}
We shall discuss three choices of spaces $(X,(\cdot,\cdot)_X)$ along with suitable time discretizations in Sections~\ref{ss:ex1}--\ref{ss:ex3} below. 

\subsubsection{Projected $L^2$-gradient flow}\label{ss:ex1}
The most popular choice $X=L^2(\Omega)$ leads to the ordinary $L^2$-gradient flow. In this case, $R_{L^2}(z)=z$ and the projection reads
 $$
P_{z,L^2}(v) = v -  \frac{(z,v)_{L^2(\Omega)}}{(z,z)_{L^2(\Omega)}} z.
$$
The $L^2$-gradient is given by the Gross-Pitaevskii differential operator. In particular, for any $z \in H^2(\Omega)$, we have
$$
\nabla_{L^2}E(z) = - \triangle z + V \hspace{2pt} z + \beta \hspace{2pt}|z|^2 z.
$$
With $A_{z} v:=- \triangle v + V \hspace{2pt} v + \beta \hspace{2pt}|z|^2 v$, the projected $L^2$-gradient flow is given by 
\begin{align}
\label{continuous-DFL2}
z^{\prime}(t)  = - A_z z + \langle A_z z,z\rangle \frac{z}{ \|z \|_{L^2(\Omega)}^2} \quad \mbox{for } t \ge 0
\end{align}
and some initial value $z(0)=z_0 \in H^1_0(\Omega) \cap H^2(\Omega)$.
This is the normalized gradient flow of \cite[Section 2.3]{BaD04}. For focusing nonlinearities, Faou and J\'ez\'equel \cite{FaT18} proved exponential $L^2$-convergence of the flow to an eigenfunction if the starting value was selected sufficiently close. This is to our best knowledge the only convergence result for the projected $L^2$-gradient flow in the context of nonlinear Schr\"odinger equations. Applying a certain first order splitting method together with a semi-implicit backward Euler discretization with step size $\tau>0$, the DNGF approach is obtained \cite{BaD04}. For the sake of consistent notation we will refer to it as GF$L^2$.
\begin{definition}[Method: GF$L^2$ (known as DNGF)]
Let $z^0 \in H^1_0(\Omega)$ be given. Then the GF$L^2$ iteration for $n\ge 0$ reads
\begin{align}
\label{GFL2-eqn}
z^{n+1} = \alpha_n (1 + \tau A_{z^n})^{-1}  z^n
\qquad \mbox{where } \qquad
\alpha_{n} := \| (1 + \tau A_{z^n})^{-1}  z^n \|_{L^2(\Omega)}^{-1}.
\end{align}
\end{definition}
By construction, the continuous flow \eqref{continuous-DFL2} is mass-conservative and energy-dissipative. However, on the time-discrete level, energy dissipation is only established for a modified energy (cf. \cite[Lemma 2.10]{BaD04}) that can be seen as an eigenvalue functional. There is no proof of global convergence in $H^1(\Omega)$ to a critical point of $E$, neither for the GF$L^2$ iteration nor the continuous flow.
However, local convergence in $H^1(\Omega)$ was established in \cite{FaT18} for focusing nonlinearities, i.e. for $\beta=-1$.

Promising computational improvements of GF$L^2$ by using sophisticated preconditioners were obtained in \cite{ALT17}.

\subsubsection{Projected $H^1$-Sobolev gradient flow}\label{ss:ex2}
In the second example, we consider the choice $X=H^1_0(\Omega)$ equipped with the standard inner product $(\nabla \cdot , \nabla \cdot)_{L^2(\Omega)}$ (cf. \cite{KaE10}). Then the Ritz-projection $R_{H^1}:L^2(\Omega) \rightarrow H^1_0(\Omega)$ is characterized by
\begin{align*}
( \nabla R_{H^1}(z) , \nabla v )_{L^2(\Omega)} = ( z , v )_{L^2(\Omega)} \qquad \mbox{for all } v \in H^1_0(\Omega).
\end{align*}
The Sobolev gradient $\nabla_{H^1}E(z) \in H^1_0(\Omega)$ is defined according to \eqref{def-sobolev-gradient} and the continuous flow reads
\begin{align}
\label{continuous-DFH1}
z^{\prime}  = - \nabla_{H^1}E(z) + \frac{( \nabla_{H^1}E(z) ,z )_{L^2(\Omega)}}{ ( R_{H^1}(z) ,z )_{L^2(\Omega)} } R_{H^1}(z)   \qquad \mbox{for } t \ge 0
\qquad \mbox{and } \qquad z(0)=z_0.
\end{align}
\cite[Thms.~5 and 6]{KaE10} reports well-posedness and exponential convergence of the flow to a critical point of $E$ in $H^1$. The discretization of the continuous flow \eqref{continuous-DFH1} using the forward Euler method leads to the GF$H^1$ approach.
\begin{definition}[Method: GF$H^1$]
Let $z^0 \in H^1_0(\Omega)$ be given. Then the GF$H^1$ iteration for $n\ge 0$ reads
\begin{align}
\label{GFH1-eqn}
\hat{z}^{n+1}  = z^n - \tau \nabla_{H^1}E(z^n) + \tau \frac{( \nabla_{H^1}E(z^n) ,z^n )_{L^2(\Omega)}}{ ( R_{H^1}(z^n) ,z^n )_{L^2(\Omega)} } R_{H^1}(z^n)
\end{align}
with the normalization $z^{n+1}=\hat{z}^{n+1} / \| \hat{z}^{n+1} \|_{L^2(\Omega)}$ after each iteration.
\end{definition}
To the best of our knowledge, there is neither a proof of energy dissipation of the GF$H^1$ iteration \eqref{GFH1-eqn} nor convergence in $H^1(\Omega)$ to a critical point of $E$. In numerical experiments, the desired properties are observed for sufficiently small $\tau$ (cf. \cite{KaE10,DaK10} or the numerical experiments in Section \ref{subsec:mod1:groundstates}).

\subsubsection{Projected $a_0$-Sobolev gradient flow}\label{ss:ex3}
In the final example we choose $X=H^1_0(\Omega)$ again, but equip it with an inner product that incorporates the potential $V$. We set $( \cdot ,\cdot )_X = a_0(\cdot,\cdot)$, where
$$
a_0( v ,w ) := (\nabla v , \nabla w )_{L^2(\Omega)} + (V \hspace{2pt} v , w )_{L^2(\Omega)}. 
$$
This choice was proposed in \cite{DaK10} in a more general setup that involves angular momentum rotation. Define the Ritz-projection $R_{a_0}(z) \in H^1_0(\Omega)$ by
$$
a_0( R_{a_0}(z) , v ) = ( z , v )_{L^2(\Omega)} \qquad \mbox{for all } v \in H^1_0(\Omega).
$$
Then, the continuous projected gradient flow reads
\begin{align}
\label{continuous-DFa0}
z^{\prime} &= - z - R_{a_0}(\beta |z|^2 z)
+ \frac{(z,z + R_{a_0}(\beta |z|^2 z) )_{L^2(\Omega)}}{(z, R_{a_0}z)_{L^2(\Omega)} } R_{a_0} z
\end{align}
completed by the initial condition $z(0)=z_0\in H^1_0(\Omega)$. Well-posedness of this gradient flow 
follows from \cite[Theorem 3.2]{DaK10}.

A forward Euler discretization leads to the following method
\begin{definition}[Method: GF$a_0$]
Let $z^0 \in H^1_0(\Omega)$ be given. Then the GF$a_0$ iteration for $n\ge 0$ reads
\begin{align}
\label{GFa0-eqn}
\hat{z}^{n+1}  = (1 -\tau) z^n - \tau R_{a_0}(\beta |z^n|^2 z^n)
+ \tau \frac{(z^n,z^n + R_{a_0}(\beta |z^n|^2 z^n) )_{L^2(\Omega)}}{(z^n, R_{a_0}z^n)_{L^2(\Omega)} } R_{a_0} z^n,
\end{align}
together with the normalization $z^{n+1}=\hat{z}^{n+1} / \| \hat{z}^{n+1} \|_{L^2(\Omega)}$.
\end{definition}
Proofs of energy reduction or the convergence of $z^n$ to a stationary point of $E$ are not available in the literature.

\section{Continuous Projected $a_z$-Soblev Gradient Flow}
\label{sec:continuous:grad:flow}
In this section we propose and analyze a new Sobolev gradient flow in $X=H^1_0(\Omega)$ based on an inner product that changes with the flow itself. For any $z\in H^1_0(\Omega)$, we define the weighted energy inner product $a_z(\cdot,\cdot)$ by  
\begin{align*}
a_z(v,w) := \int_{\Omega} \nabla v \cdot \nabla w + V \hspace{2pt} v \hspace{2pt} w + \beta \hspace{2pt}|z|^2 v \hspace{2pt} w\dx
\end{align*}
for $v,w\in H^1_0(\Omega)$. Since $\langle E^{\prime}(z) , v \rangle = a_z(z,v)$ for any $z\in H^1_0(\Omega)$, the Sobolev gradient of $E(z)$ with respect to $a_z(\cdot,\cdot)$ is the identity, i.e., $\nabla_{a_z} E(z) = z$. 
The gradient flow of $E$ with respect to $\nabla_{a_z}$ projected into the tangent space associated with mass constraint $\int_{\Omega}|z|^2\dx=1$ is thus characterized by
$$
z^{\prime}(t) = - P_{z(t)} \nabla_{a_{z(t)}} E(z(t)) = - P_{z(t)} (\hspace{1pt}z(t)\hspace{1pt}).
$$
The projection $P_{z}$ can be written as
\begin{align}
\label{def-Pz}P_z(h) = h - \frac{(z,h)_{L^2(\Omega)}}{(z,\mathcal{G}_z z)_{L^2(\Omega)} } \mathcal{G}_z z, 
\end{align}
where $\mathcal{G}_z$ is just the Green's operator (or Ritz-projection) associated with the time-dependent inner product $a_z(\cdot,\cdot)$, i.e., for any $f \in L^2(\Omega)$, $\mathcal{G}_z(f) \in H^1_0(\Omega)$ satisfies
\begin{align*}
a_z( \hspace{2pt} \mathcal{G}_z(f) ,  v ) = (f ,v)_{L^2(\Omega)}
\end{align*}
for all $v \in H^1_0(\Omega)$. 
Altogether, this yields the following projected gradient flow problem. 
\begin{definition}[Projected $a_z$-Sobolev Gradient Flow]\label{def-weighted-gradient-flow}
Given $z_0 \in H^1_0(\Omega)$ with $\| z_0 \|_{L^2(\Omega)}=1$, find a differentiable function 
$z \in C^1([0,\infty);H^1_0(\Omega))$
with $z(0)=z_0$ such that, for all $t>0$,
\begin{align}
\label{sobolev-flow-az}
z^{\prime}(t) = - z(t) + \gamma_{z(t)} \mathcal{G}_{z(t)} z(t), \qquad \mbox{where } 
\gamma_{z} :=  \frac{(z,z)_{L^2(\Omega)}}{a_{z}(\mathcal{G}_{z} z ,\mathcal{G}_{z} z)} >0.
\end{align}
\end{definition}
The subsequent theorem states the well-posedness of this flow and all its  important properties.
\begin{theorem}\label{main-result-1-cont-sob-grad}
For any initial value $z_0 \in H^1_0(\Omega)$ with $\| z_0 \|_{L^2(\Omega)}=1$, there exists a unique global solution $z$ to the Sobolev gradient flow problem stated in Definition \ref{def-weighted-gradient-flow}. The flow is mass-conservative, i.e., $\| z(t) \|_{L^2(\Omega)}=1$ for all $t\in\mathbb{R}_{\ge0}$, and energy-dissipative, i.e. $E(\hspace{1pt}z(t)\hspace{1pt}) \le E(\hspace{1pt}z(s)\hspace{1pt})$ for all $0\le s \le t < \infty$. Moreover, $z$ converges globally in $H^1(\Omega)$ to an eigenfunction $z^{\ast}$ with eigenvalue $\lambda^{\ast}=\gamma_{z^{\ast}}$ of the Gross-Pitaevskii equation~\eqref{weak-problem}. 

If $z^{\ast}=z_{\mbox{\rm\tiny GS}}$ is the unique positive ground state eigenfunction from Proposition \ref{basic-existence-ground-state} (which can be obtained by selecting $z_0$ as a nonnegative functions) then the convergence rate is asymptotically exponential in the following sense. For all $0<\delta \le 1$,  there exists some $c_{\delta}>0$ and a finite time $0<t_{\delta}<\infty$, such that for all $t\ge t_{\delta}$ 
\begin{align*}
\left(\|\nabla z_{\mbox{\rm\tiny GS}} - \nabla z(t)\|_{L^2(\Omega)}^2 + \|\sqrt{V}(z_{\mbox{\rm\tiny GS}} - z(t))\|_{L^2(\Omega)}^2\right)^{1/2}
\le \hspace{2pt} c_{\delta} \operatorname{exp}\left(-\left(1 - \frac{\lambda_{\mbox{\rm\tiny GS}}}{\mu_2}  -\delta \right)t\right). 
\end{align*}
Here, $\lambda_{\mbox{\rm\tiny GS}}>0$ is the ground state eigenvalue and $\mu_2>\lambda_{\mbox{\rm\tiny GS}}$ is the second eigenvalue of the linearized eigenvalue problem seeking $w_i \in  H^1_0(\Omega)$ with $\| w_i \|_{L^2(\Omega)}=1$ and $\mu_i \in \R$ such that
$$
a_{z_{\mbox{\rm\tiny GS}}}( w_i , v ) = \mu_i ( w_i , v )_{L^2(\Omega)} \qquad \mbox{for all } v \in H^1_0(\Omega). 
$$
\end{theorem}
The remainder of this section is devoted to the proof of the theorem.

\subsection{Energy decay, mass conservation and local well-posedness}
This subsection shows that a well-defined flow $z$ is energy-diminishing and mass-conserving, as expected.
\begin{lemma}[Mass conservation and energy reduction]\label{lemma-mass-con-energy-red}
We consider the weighted Sobolev gradient flow of Definition \ref{def-weighted-gradient-flow}. If $z(t)$ is well-defined on an interval $[0,T)$ for some $T>0$ then, for all $0\le t \le t' < T$,
\begin{align*}
\| z(t) \|_{L^2(\Omega)} = 1 \quad \text{and} \quad
E( \hspace{1pt} z(t') \hspace{1pt} ) \le E( \hspace{1pt} z(t) \hspace{1pt} ).
\end{align*} 
\end{lemma}
\begin{proof}
Let $t\in [0,T)$ be arbitrary but fixed. Noting $a_z(\mathcal{G}_{z} z,\mathcal{G}_{z} z)= (z,\mathcal{G}_{z} z)_{L^2(\Omega)}$ and testing in the $L^2$-variational formulation of \eqref{sobolev-flow-az} with $z$ yields
\begin{align*}
\tfrac{1}{2} \frac{\mbox{d}}{\mbox{d}t} \| z \|_{L^2(\Omega)}^2 
= - ( z , z )_{L^2(\Omega)} + \frac{(z,z)_{L^2(\Omega)}}{(z,\mathcal{G}_{z} z)_{L^2(\Omega)} } (\mathcal{G}_z z , z )_{L^2(\Omega)} = 0.
\end{align*}
This implies conservation of mass. By definition, $\mathcal{G}_z z \in H^1_0(\Omega)$ and, hence,
$$
z^{\prime}(t) = - z(t) + \gamma_{z(t)} \mathcal{G}_{z(t)} z(t) \in H^1_0(\Omega).
$$ 
We can therefore use $z^{\prime}$ as a test function in the energy-inner product to obtain
\begin{align*}
a_z( z^{\prime} , z^{\prime} ) = - a_z( z , z^{\prime} ) + \frac{(z,z)_{L^2(\Omega)}}{(z,\mathcal{G}_{z} z)_{L^2(\Omega)} }  a_z ( \mathcal{G}_z z , z^{\prime} ).
\end{align*}
This implies 
$$
a_z( z , z^{\prime} ) = \langle E^{\prime}(z) , z^{\prime} \rangle 
 = \frac{\mbox{d}}{\mbox{d}t} E(z) \quad \mbox{and} \quad
 a_z ( \mathcal{G} z , z^{\prime} ) = ( z , z^{\prime} )_{L^2(\Omega)} = \tfrac{1}{2} \frac{\mbox{d}}{\mbox{d}t} \| z \|_{L^2(\Omega)}^2  = 0.$$ 
 The combination of the previous equalities readily yields
\begin{align*}
0 \le a_{z(t)}( z^{\prime}(t) , z^{\prime}(t) ) = - \frac{\mbox{d}}{\mbox{d}t} E(\hspace{1pt}z(t)\hspace{1pt}),
\end{align*}
which shows that energy is reduced along the flow. 
\end{proof}
To prove local existence of $z(t)$ in a neighborhood of $z_0$ for some maximal time $T>0$, we define the bounded and closed set $H_{0,M}^1(\Omega)$ for a given $M>0$ by
$$
H_{0,M}^1(\Omega):= \left\{ v \in H^1_0(\Omega)| \hspace{2pt} \| v - z_0 \|_{H^1(\Omega)} \le M \right\}.
$$
We need to show that the operator that describes the right-hand side of the flow problem \eqref{sobolev-flow-az} is Lipschitz-continuous on $H_{0,M}^1(\Omega)$. For the ease of notation we set $\mathcal{K}z:=\mathcal{G}_{z} z$.

\begin{lemma}\label{continuity-calK}
The operator $\mathcal{K} : H^1_0(\Omega) \rightarrow H^1_0(\Omega)$ is Lipschitz-continuous and bounded on $H_{0,M}^1(\Omega)$. In particular, there exists a constant $L_M>0$ that depends on $M$, $z_0$, $\beta$, $\Omega$ and $d$, such that
\begin{align*}
\| \mathcal{K}(v) - \mathcal{K}(w) \|_{H^1(\Omega)}
\le L_M \| v - w \|_{H^1(\Omega)}
\qquad \mbox{for all } v,w \in H_{0,M}^1(\Omega).
\end{align*}
\end{lemma}

\begin{proof}
Let $v,w \in H_{0,M}^1(\Omega)$ and set $M_0:= M + \| z_0 \|_{H^1(\Omega)}$. Since
$$
\| \nabla \mathcal{K} v \|_{L^2(\Omega)}^2
\le a_v(  \mathcal{K}v , \mathcal{K} v ) = ( v , \mathcal{K}v )_{L^2(\Omega)},
$$
we conclude that \begin{align}
\label{bound-for-calK}\|  \mathcal{K} v  \|_{H^1(\Omega)} \le C \| v \|_{L^2(\Omega)} \le C M_0, 
\end{align}
where $C$ only depends on the Poincar\'e-Friedrichs constant. With this, we have
\begin{eqnarray*}
\lefteqn{\| \mathcal{K} v - \mathcal{K} w \|_{H^1(\Omega)}^2 \le  
a_0( \mathcal{K} v -\mathcal{K} w, \mathcal{K} v - \mathcal{K} w ) + \beta (|v|^2 (\mathcal{K} v - \mathcal{K} w ) , (\mathcal{K} v -\mathcal{K} w ) )_{L^2(\Omega)}} \\
&=&
a_0( \mathcal{K} v , \mathcal{K} v -\mathcal{K} w ) + \beta (|v|^2 \mathcal{K} v , (\mathcal{K} v -\mathcal{K} w ) )_{L^2(\Omega)} \\
&\enspace&\quad - a_0(\mathcal{K} w, \mathcal{K}v - \mathcal{K} w ) - \beta (|v|^2 \mathcal{K} w  , (\mathcal{K}v -\mathcal{K} w ) )_{L^2(\Omega)} \\
&=& ( v - w,  \mathcal{K} v - \mathcal{K} w )_{L^2(\Omega)} 
+ \beta ((|w|^2 - |v|^2) \mathcal{K} w  , (\mathcal{K} v - \mathcal{K} w ) )_{L^2(\Omega)} \\
&\lesssim& \| v - w \|_{H^1(\Omega)} \|  \mathcal{K} v - \mathcal{K} w \|_{H^1(\Omega)} 
+ \beta ((|w|^2 - |v|^2) \mathcal{K} w  , (\mathcal{K} v - \mathcal{K} w ) )_{L^2(\Omega)}.\hspace{30pt}
\end{eqnarray*}
Using the H\"older inequality and embedding estimates, the second term on the right-hand side can be bounded by
\begin{eqnarray*}
\lefteqn{|((|w|^2 - |v|^2) \mathcal{K} w  , (\mathcal{K} v - \mathcal{K} w ) )_{L^2(\Omega)}|} \\
&\le& \left( \| w \|_{L^4(\Omega)} + \| v \|_{L^4(\Omega)}  \right) \| \mathcal{K} w \|_{L^4(\Omega)}
\| v - w \|_{L^4(\Omega)} 
\| \mathcal{K} v - \mathcal{K} w \|_{L^4(\Omega)} \\
&\lesssim&
\left( \| w \|_{H^1(\Omega)} + \| v \|_{H^1(\Omega)}  \right) \| \mathcal{K} w \|_{H^1(\Omega)}
\| v - w \|_{H^1(\Omega)} 
\| \mathcal{K} v - \mathcal{K} w \|_{H^1(\Omega)}.
\end{eqnarray*}
Using \eqref{bound-for-calK} and $v,w \in H_{0,M}^1(\Omega)$ we conclude the existence of some $L_M=\mathcal{O}(1+ \beta M_0^2)$ such that the Lipschitz-continuity holds true.
\end{proof}
The Lipschitz-continuity of Lemma~\ref{continuity-calK} and the trivial observation that $\mathcal{K}v=0$ if and only if $v=0$ imply that there exists a sufficiently small neighbourhood $H_{0,M}^1(\Omega)$ of $z_0$ and a constant $c_M>0$ such that 
$$
\| \mathcal{K}v \|_{H^{1}(\Omega)} \ge c_M \qquad \mbox{for all } v \in H_{0,M}^1(\Omega).
$$
In such a neighborhood we have that the normalization factor $\gamma_{v}$, as a function in $v$, is also Lipschitz-continuous.
\begin{lemma}\label{gamma-continuity}
For any sufficiently small $M>0$ the functional $\gamma_{\cdot} : H^1_0(\Omega) \rightarrow \mathbb{R}$ is Lipschitz-continuous and bounded on $H_{0,M}^1(\Omega)$. In particular, there is $\tilde{L}_M>0$ (depending only on $M$, $z_0$, $\beta$, $V$, $\Omega$ and $d$) such that
\begin{align*}
|\gamma_v - \gamma_w |
\le \tilde{L}_M \| v - w \|_{H^1(\Omega)}
\qquad \mbox{for all } v,w \in H_{0,M}^1(\Omega).
\end{align*}
\end{lemma}
\begin{proof}
The error in the difference of $\gamma_v$ and $\gamma_w$ can be expressed as 
\begin{align*}
\left| \frac{(v,v)_{L^2(\Omega)}}{ a_v( \mathcal{K} v , \mathcal{K} v ) } 
-  \frac{(w,w)_{L^2(\Omega)}}{ a_w( \mathcal{K} w , \mathcal{K} w )  }  \right|
&= \left| \frac{(v,v)_{L^2(\Omega)} \hspace{2pt} a_w( \mathcal{K} w , \mathcal{K} w )
-
(w,w)_{L^2(\Omega)} \hspace{2pt} a_v( \mathcal{K} v , \mathcal{K} v ) }{
a_v( \mathcal{K} v , \mathcal{K} v ) \hspace{2pt} a_w( \mathcal{K} w , \mathcal{K} w )
 }  \right|
\end{align*}
and
\begin{align*}
\left| \gamma_v - \gamma_w  \right|
&\le c_M^{-4}  \left| (v,v)_{L^2(\Omega)} \hspace{2pt} a_w( \mathcal{K} w , \mathcal{K} w )
-
(w,w)_{L^2(\Omega)} \hspace{2pt} a_v( \mathcal{K} v , \mathcal{K} v )  \right|.
\end{align*}
Using the splitting
\begin{eqnarray*}
\lefteqn{ (v,v)_{L^2(\Omega)} \hspace{2pt} a_w( \mathcal{K} w , \mathcal{K} w )
-
(w,w)_{L^2(\Omega)} \hspace{2pt} a_v( \mathcal{K} v , \mathcal{K} v ) } \\
&=&  \left( \| v \|^2_{L^2(\Omega)} - \| w \|^2_{L^2(\Omega)}  \right) \hspace{2pt} a_w( \mathcal{K} w , \mathcal{K} w )\\
&\enspace&\quad+
 \| w \|^2_{L^2(\Omega)} \left( a_w( \mathcal{K} w , \mathcal{K} w ) - a_w( \mathcal{K} v , \mathcal{K} v ) \right)
+ \beta \| w \|^2_{L^2(\Omega)} \int_{\Omega} (|w|^2 -|v|^2)  | \mathcal{K} v|^2\dx.
\end{eqnarray*}
and the norm inequality 
$\| a \|^2 - \| b \|^2 \le \| a - b \| ( \| a \| + \| b \| )$ (which follows from the inverse triangle inequality)
we see that there exists a constant $\tilde{C}_M>0$ such that for all $v,w\in H_{0,M}^1(\Omega)$ it holds
\begin{align*}
\left| \gamma_v - \gamma_w  \right|
&\le \tilde{C}_M \left( \| v - w \|_{H^1(\Omega)} + \| \mathcal{K} v -  \mathcal{K} w \|_{H^1(\Omega)} \right).
\end{align*}
The Lipschitz-continuity of $ \mathcal{K}$ on $H_{0,M}^1(\Omega)$ as shown in  Lemma~\ref{continuity-calK} finishes the proof.
\end{proof}

The combination of Lemmas \ref{continuity-calK} and \ref{gamma-continuity} shows that the right-hand side  $g(z):=-z + \gamma_z  \mathcal{G}_{z} z$ of the gradient flow problem of Definition~\ref{definition-GF-az} is Lipschitz-continuous in the neighborhood $H_{0,M}^1(\Omega)$ of $z_0$. Thus, the classical Picard-Lindel\"of Theorem for Hilbert spaces implies local existence and uniqueness for some time $T>0$.
\begin{lemma}[Local well-posedness]\label{local-existence-cont-sob-grad}
For any $z_0 \in H^1_0(\Omega)$ with $\| z_0 \|_{L^2(\Omega)}=1$, there exists a maximum time $T>0$ such that there is a unique solution $z$ to \eqref{sobolev-flow-az} on the time interval $[0,T)$.
\end{lemma}

\subsection{Global well-posedness}
Starting from the local existence of $z(t)$ guaranteed by Lemma~\ref{local-existence-cont-sob-grad}, Lemma \ref{lemma-mass-con-energy-red} allows us to pass to a global existence result. Note that as soon as such a global existence result is established, Lemma \ref{lemma-mass-con-energy-red} implies mass conservation and energy-reduction for all times $t\in [0,\infty)$.

\begin{proof}[Proof of Theorem \ref{main-result-1-cont-sob-grad} - Global well-posedness]
Recall that Lemma~\ref{local-existence-cont-sob-grad} guarantees the existence of a unique solution on a time interval $[0,T)$ and assume that $T$ is finite and maximal in the sense that the problem is no longer well-posed for $t\ge T$. The energy reduction in Lemma \ref{lemma-mass-con-energy-red} guarantees that $E_T := \lim_{t \rightarrow T}E(z(t))$ exists. Next, let $t_1,t_2 \in [0,T)$ be arbitrary with $t_1\le t_2$. Then, using the estimate 
$\| \int_{t_1}^{t_2} v(t) \hspace{2pt} dt \|_{H^1(\Omega)} \le \int_{t_1}^{t_2} \| v(t) \|_{H^1(\Omega)} \hspace{2pt} dt$ (cf. \cite[Appendix E.5]{Eva10}) and the construction of $z$ we see that
\begin{eqnarray*}
\lefteqn{\| z(t_2) - z(t_1) \|_{H^1(\Omega)}^2 =\| \int_{t_1}^{t_2} z^{\prime}(t) \hspace{2pt} dt \|_{H^1(\Omega)}^2
\le \left( \int_{t_1}^{t_2} \| z^{\prime}(t)  \|_{H^1(\Omega)} \dt \right)^2} \\
&\le& (t_2 - t_1) \left( \int_{t_1}^{t_2} a_{z(t)}( z^{\prime}(t) , z^{\prime}(t) ) \dt \right)
= (t_2 - t_1) ( E(z(t_2)) - E(z(t_1)) ).
\end{eqnarray*}
This implies boundedness of $z(t)$ in $H^1_0(\Omega)$. Hence, we have the existence of a sequence $\{ t^{n} \}_{n\in \mathbb{N}}$ with $t^n \rightarrow T$ and a function $z_T\in H^1_0(\Omega)$ so that $z(t^{n})\rightharpoonup z_T$ weakly in $H^1_0(\Omega)$. This implies
\begin{align*}
\| z(t^n) \|_{H^1(\Omega)} &\le  \| z_T \|_{H^1(\Omega)}  + \| z(t^n) -  z_T \|_{H^1(\Omega)} \\
&\le  \| z_T \|_{H^1(\Omega)}  + \underset{m \rightarrow \infty}{\mbox{\rm lim inf }}\| z(t^n) -  z(t^m) \|_{H^1(\Omega)}  \\
&\le  \| z_T \|_{H^1(\Omega)}  + \underset{m \rightarrow \infty}{\mbox{\rm lim inf }} \sqrt{(t^m - t^n) ( E(z(t^n)) - E(z(t^m)) )} \\
&=  \| z_T \|_{H^1(\Omega)}  +  \sqrt{(T - t^n) ( E(z(t^n)) - E_T )}.
\end{align*}
Consequently
$$
\lim_{n\rightarrow {\infty}} \| z(t^n) \|_{H^1(\Omega)} \le \| z_T \|_{H^1(\Omega)} + \lim_{n\rightarrow \infty} \sqrt{(T - t^n) ( E(z(t^n)) - E_T) )} = \| z_T \|_{H^1(\Omega)}.
$$
Since the Hilbert space $H^1_0(\Omega)$ is uniformly convex, the weak convergence together with $\lim_{n\rightarrow {\infty}} \| z(t^n) \|_{H^1(\Omega)} \le \| z_T \|_{H^1(\Omega)}$ guarantee that $z(t^{n})\rightarrow z_T$ strongly in $H^1_0(\Omega)$. The continuity of $z$ in $t$ implies independence of this limit on the choice of the sequence $t^n$. Consequently, we have $z(t)\rightarrow z(T) := z_T$ for $t\rightarrow T$, strongly in $H^1_0(\Omega)$. Since we assumed $T<\infty$, we could use $z(T)$ as a new starting value to guarantee existence of $z$ on an extended interval $[0,T+\delta)$ for some $\delta>0$. This contradicts the assumed maximality of $T$ and, hence, shows that the problem admits a unique solution for all times.
\end{proof}

\subsection{Global convergence and exponential decay to the ground state}
With the previous results, we can now prove global $H^1$-convergence of $z(t)$ to a critical point of $E$ that fulfills the normalization constraint. In a first step, we need to make an identification of the limit.
\begin{proof}[Proof of Theorem \ref{main-result-1-cont-sob-grad} - Limit is eigenfunction of the GPE]
It remains to identify the limit as an eigenfunction of the GPE problem. Since the (non-negative) energy is decreasing along the flow there exists some limit
$E_{\infty}:=\lim_{t\rightarrow {\infty}} E(z(t))$. With $a_{z(t)}( z^{\prime}(t) , z^{\prime}(t) ) = - \frac{\mbox{\scriptsize d}}{\mbox{\scriptsize d}t} E(\hspace{1pt}z(t)\hspace{1pt})$  we can conclude that
\begin{align*}
\int_{0}^{\infty} a_{z(t)}( z^{\prime}(t) , z^{\prime}(t) ) \dt = E(\hspace{1pt}z(0)\hspace{1pt}) - E_{\infty} < \infty.
\end{align*}
This implies that $\int_{0}^{\infty} \| z^{\prime} \|^2_{H^1(\Omega)}\dt$ is finite and there exists $z^{\ast}\in H^1_0(\Omega)$
such that $z(t)\rightarrow z^{\ast}$ strongly in $H^1(\Omega)$. Obviously, the limit fulfills $z^{\ast} = \gamma_{z^{\ast}} \mathcal{G}_{z^{\ast}} z^{\ast}$ and, hence,
$$
a_{z^{\ast}}( z^{\ast} , v ) = \gamma_{z^{\ast}} \hspace{1pt}a_{z^{\ast}}( \mathcal{G}_{z^{\ast}} z^{\ast} , v ) = \gamma_{z^{\ast}} ( z^{\ast} , v)_{L^2(\Omega)},
$$
i.e, $(z^{\ast},\gamma_{z^{\ast}})$ is an eigensolution of \eqref{weak-problem}.
\end{proof}
We are now ready to prove the exponential convergence to the ground state.
\begin{proof}[Proof of Theorem \ref{main-result-1-cont-sob-grad} - Exponential convergence to ground state]
Assume that the strong $H^1$-limit of the flow $z(t)$ coincides with the unique positive ground state, i.e. $z^{\ast}=z_{\mbox{\rm\tiny GS}}$, where  $z_{\mbox{\rm\tiny GS}} >0$ is characterized as in Proposition \ref{basic-existence-ground-state}. From the first part of the proof of Theorem \ref{main-result-1-cont-sob-grad} we also know that $\gamma_{z}$ converges to the ground state eigenvalue $\lambda^{\ast}=\lambda_{\mbox{\rm\tiny GS}}>0$ for $t \rightarrow \infty$.

The proof of exponential convergence is based on a Gr\"onwall-type argument. For that, we define the function
\begin{align*}
f(t) := \tfrac{1}{2} a_{z(t)}( z^{\prime}(t), z^{\prime}(t) ),
\end{align*}
and want to show that $f^{\prime} \le c \hspace{2pt} f$ for some positive constant $c$.
Since $2\int_{0}^{\infty} f(t) \dt = E(z_0) -E_{\infty}$ is finite, we know that $f(t)\rightarrow 0$ for $t\rightarrow \infty$ and 
hence $f(t)^{3/2} \le f(t)$ for all sufficiently large times. Using this fact, we can conclude that for any $\eps_0>0$, there exists a finite time $ t(\eps_0)\ge0$ such that for all $t\ge t(\eps_0)$ it holds
\begin{align}
\label{f-asymptotic-delta}
f(t)^{3/2} \le \hspace{1pt}\eps_0 \hspace{1pt} f(t).
\end{align}
Next, recall the projection $P_z$ from \eqref{def-Pz}. With $P_z(z)= z - \gamma_z \mathcal{G}_z z =-z^{\prime}$, we can rewrite 
$f=\tfrac{1}{2} a_{z}( P_z(z) , P_z(z) )$ 
and, hence,
\begin{align}
\label{f-tilde-prime}
f^{\prime} = a_z( \frac{\mbox{d}}{\mbox{d}t} P_z(z), P_z(z) ) +  \frac{1}{2} a_z^{\prime}\langle z^{\prime} , (P_z(z) , P_z(z))  \rangle 
\end{align}
where $\frac{\mbox{d}}{\mbox{d}t} P_z(z)$ is given by
\begin{align}
\label{Frechet-derivative-G}
\frac{\mbox{d}}{\mbox{d}t} P_z(z)= z^{\prime} - \left\langle  \frac{\mbox{d}}{\mbox{dz}} \left( \gamma_z \mathcal{G}_z z \right) , z^{\prime} \right\rangle \in H^1_0(\Omega)
\end{align}
and where $a_z^{\prime} \langle v, (w_1,w_2) \rangle$ is the Fr\'echet derivative of $a_z(\cdot,\cdot)$ wrt. $z$, which can be computed as
$$
a_z^{\prime} \langle v, (w_1,w_2) \rangle = 2 \beta \int_{\Omega} z v w_1 w_2 \dx.
$$
Consequently, we have with \eqref{f-tilde-prime}
\begin{align}
\label{f-tilde-prime-2}
\nonumber
f^{\prime} &= 
- a_z(z^{\prime},z^{\prime}) +
a_z( \left\langle  \frac{\mbox{d}}{\mbox{dz}} \left( \gamma_z \mathcal{G}_z z \right) , z^{\prime} \right\rangle , z^{\prime} )
+   \beta \int_{\Omega} z (z^{\prime})^3 \dx \\
\nonumber&= - 2 f 
+ 
\left\langle  \frac{\mbox{d}}{\mbox{dz}} \left( \gamma_z \right) , z^{\prime} \right\rangle 
a_z( \mathcal{G}_z z   , z^{\prime} )
+
\gamma_z a_z( \left\langle  \frac{\mbox{d}}{\mbox{dz}} \left( \mathcal{G}_z z \right) , z^{\prime} \right\rangle , z^{\prime} )
+   \beta \int_{\Omega} z (z^{\prime})^3 \dx \\
\nonumber&= - 2 f 
+ 
\left\langle  \frac{\mbox{d}}{\mbox{dz}} \left( \gamma_z \right) , z^{\prime} \right\rangle 
( z   , z^{\prime} )_{L^2(\Omega)}
+
\gamma_z a_z( \left\langle  \frac{\mbox{d}}{\mbox{dz}} \left( \mathcal{G}_z z \right) , z^{\prime} \right\rangle , z^{\prime} )
+   \beta \int_{\Omega} z (z^{\prime})^3 \dx\\
\nonumber&= - 2 f 
+
\gamma_z a_z( \left\langle  \frac{\mbox{d}}{\mbox{dz}} \left( \mathcal{G}_z z \right) , z^{\prime} \right\rangle , z^{\prime} )
+   \beta \int_{\Omega} z (z^{\prime})^3 \dx \\
\nonumber&\le - 2 f 
+
\gamma_z a_z( \left\langle  \frac{\mbox{d}}{\mbox{dz}} \left( \mathcal{G}_z z \right) , z^{\prime} \right\rangle , z^{\prime} ) + \beta \| z^{\prime} \|^3_{L^6(\Omega)} \\
&\le - 2 f 
+
\gamma_z a_z( \left\langle  \frac{\mbox{d}}{\mbox{dz}} \left( \mathcal{G}_z z \right) , z^{\prime} \right\rangle , z^{\prime} ) + C(\Omega,V,\beta,z_0)  f^{3/2},
\end{align}
where we used the Sobolev embedding $H^1_0(\Omega) \hookrightarrow L^6(\Omega)$ (for $d\le3$) in the last step. 
Next, we investigate the middle term. We use
\begin{align*}
( \cdot , w )_{L^2(\Omega)} = \langle \frac{\mbox{d}}{\mbox{d}z} a_z(  \mathcal{G}_z z , w ) , \cdot \rangle
= a_z^{\prime}\langle  \cdot, ( \mathcal{G}_z z , w) \rangle +  a_z( \left\langle \frac{\mbox{d}}{\mbox{d}z}  \mathcal{G}_z z , \cdot \right\rangle , w ) 
\end{align*}
to see that
\begin{align*}
a_z( \left\langle  \frac{\mbox{d}}{\mbox{dz}} \left( \mathcal{G}_z z \right) , z^{\prime} \right\rangle , z^{\prime} )
= \| z^{\prime} \|_{L^2(\Omega)}^2 - a_z^{\prime}\langle  z^{\prime}, ( \mathcal{G}_z z , z^{\prime} ) \rangle.
\end{align*}
The latter term can be written as
\begin{eqnarray*}
\lefteqn{a_z^{\prime}\langle  z^{\prime}, ( \mathcal{G}_z z , z^{\prime} ) \rangle}\\
&=& 2 \beta \int_{\Omega} |z^{\prime}|^2  (z-z_{\mbox{\rm\tiny GS}}) \hspace{2pt} \mathcal{G}_z z + |z^{\prime}|^2 z_{\mbox{\rm\tiny GS}} \hspace{2pt} (\mathcal{G}_z z - \mathcal{G}_{z_{\mbox{\rm\tiny GS}}} z_{\mbox{\rm\tiny GS}})  + |z^{\prime}|^2 z_{\mbox{\rm\tiny GS}} \hspace{2pt} \mathcal{G}_{z_{\mbox{\rm\tiny GS}}} z_{\mbox{\rm\tiny GS}} \dx ,
\end{eqnarray*}
where the first two terms are of higher order (due to the strong $H^1$-convergence to $z_{\mbox{\rm\tiny GS}}$) and the last term is strictly positive. Hence, for any $\eps_1>0$ and sufficiently large times $t\ge t(\eps_1)$, we have the crude estimate
\begin{align}
\label{crdest}
a_z( \left\langle  \frac{\mbox{d}}{\mbox{dz}} \left( \mathcal{G}_z z \right) , z^{\prime} \right\rangle , z^{\prime} )
= \| z^{\prime} \|_{L^2(\Omega)}^2 - a_z^{\prime}\langle  z^{\prime}, ( \mathcal{G}_z z , z^{\prime} ) \rangle \le \| z^{\prime} \|_{L^2(\Omega)}^2 + \eps_1 \hspace{2pt} f.
\end{align}
The crucial estimate is now for $\| z^{\prime} \|_{L^2(\Omega)}^2$. First, we note that
\begin{align}
\label{equality-zprime-az}
a_z( z^{\prime} , z^{\prime} ) 
& = - a_z( z , z^{\prime} )  = a_z( z , z )  - \gamma_z  a_z( z , \mathcal{G}_{z} z  )  = a_z( z , z )  - \gamma_z  (  z , z )_{L^2(\Omega)}  \rightarrow 0
\end{align}
for $t\rightarrow \infty$, which shows that $z^{\prime}(t)$ converges strongly in $H^1$ to zero. Furthermore, we have
\begin{align*}
\frac{a_z(z^{\prime},z^{\prime})}{ \| z^{\prime} \|_{L^2(\Omega)}^2} \ge 
\frac{ \| \nabla z^{\prime} \|_{L^2(\Omega)} }{ \| z^{\prime} \|_{L^2(\Omega)}^2} \ge C_{\Omega} > 0,
\end{align*}
where $C_{\Omega}$ is the Poincar\'e-Friedrichs constant on $\Omega$. 
We want to derive a sharper estimate for 
$$
C_{\inf} := \underset{t\rightarrow \infty}{\mbox{lim inf}} \hspace{2pt} \frac{a_z(z^{\prime},z^{\prime})}{ \| z^{\prime} \|_{L^2(\Omega)}^2}.
$$
Assume that $(t_n)_{n\in \mathbb{N}}$ with $t_n\rightarrow \infty$ is a corresponding minimal sequence so that $C_{\inf}$ is reached. In this case we have
\begin{align}
\label{zprimelowerbound}
\lim_{n\rightarrow \infty} \frac{a_z(z^{\prime}(t_n),z^{\prime}(t_n))}{ \| z^{\prime}(t_n) \|_{L^2(\Omega)}}  =C_{\inf} \| z^{\prime}(t_n) \|_{L^2(\Omega)} \rightarrow 0.
\end{align}
Note that $z^{\prime}(t_n)$ is also bounded in $H^1(\Omega)$ and hence we can assume without loss of generality that there exists a weak limit $\hat{z} \in H^1_0(\Omega)$ with $\| \hat{z} \|_{L^2(\Omega)} =1$ such that for $n\rightarrow \infty$
$$
\frac{z^{\prime}(t_n)}{ \| z^{\prime}(t_n) \|_{L^2(\Omega)}} \rightharpoonup \hat{z}
\quad \mbox{weakly in } H^1(\Omega).
$$
Together with the strong $H^1$-convergence of $z(t)$ to the ground state, this implies for $n\rightarrow \infty$
\begin{eqnarray*}
\lefteqn{\lambda_{\mbox{\rm\tiny GS}} (z_{\mbox{\rm\tiny GS}} , \hat{z} )_{L^2(\Omega)} = a_{z_{\mbox{\rm\tiny GS}}}(z_{\mbox{\rm\tiny GS}} , \hat{z} ) \longleftarrow  a_{z(t_n)}( z(t_n) , \frac{z^{\prime}(t_n)}{ \| z^{\prime}(t_n) \|_{L^2(\Omega)}}  )}\\
&\overset{\eqref{equality-zprime-az}}{=}&
- a_{z(t_n)}( z^{\prime}(t_n), \frac{z^{\prime}(t_n)}{ \| z^{\prime}(t_n) \|_{L^2(\Omega)}}  )
\overset{\eqref{zprimelowerbound}}{\longrightarrow} 0.\hspace{100pt}
\end{eqnarray*}
Hence, the function $\hat{z}$ is orthogonal to the ground state $z_{\mbox{\rm\tiny GS}} $ both with respect to the $L^2$- and the $a_{z_{\mbox{\rm\tiny GS}}}(\cdot,\cdot)$-inner product. We conclude (with the lower semi-continuity of weakly converging sequences) that
$$
C_{\inf} =\lim_{n\rightarrow \infty} \frac{a_z(z^{\prime}(t_n),z^{\prime}(t_n))}{ \| z^{\prime}(t_n) \|^2_{L^2(\Omega)}} 
\ge a_{z_{\mbox{\rm\tiny GS}}}(\hat{z} , \hat{z})
\ge \inf_{v \hspace{1pt} \in \hspace{1pt} \mbox{\footnotesize span}\{ z_{\mbox{\rm\tiny GS}} \}^{\perp} } \frac{a_{z_{\mbox{\rm\tiny GS}}}(v,v)}{ \| v \|^2_{L^2(\Omega)}}
$$
where $\mbox{span}\{ z_{\mbox{\rm\tiny GS}} \}^{\perp}$ is the $a_{z_{\mbox{\rm\tiny GS}}}(\cdot,\cdot)$-orthogonal complement of the first eigenspace. Hence, with the Courant-Fischer Theorem we have 
$$
C_{\inf} \ge \inf_{v \hspace{1pt} \in \hspace{1pt} \mbox{\footnotesize span}\{ z_{\mbox{\rm\tiny GS}} \}^{\perp} } \frac{a_{z_{\mbox{\rm\tiny GS}}}(v,v)}{ \| v \|^2_{L^2(\Omega)}} \ge \mu_2,
$$
where $\mu_2>\mu_1:=\lambda_{\mbox{\rm\tiny GS}}$ is the second eigenvalue of linear eigenvalue problem: find $w_i \in  H^1_0(\Omega)$ with $\| w_i \|_{L^2(\Omega)}=1$ and $\mu_i \in \R$
$$
a_{z_{\mbox{\rm\tiny GS}}}( w_i , v ) = \mu_i ( w_i , v )_{L^2(\Omega)} \qquad \mbox{for all } v \in H^1_0(\Omega). 
$$
Here we exploited that $\mu_1=\lambda_{\mbox{\rm\tiny GS}}$ is the smallest eigenvalue of the linearized problem and that it is also simple (cf. \cite[Lemma 2]{CCM10}). We can summarize that 
$$
 \underset{t\rightarrow \infty}{\mbox{lim inf}} \hspace{2pt} \frac{a_z(z^{\prime},z^{\prime})}{ \| z^{\prime} \|_{L^2(\Omega)}^2} \ge \mu_2
$$
and hence, for all $\eps_2>0$ there is a sufficiently large time $t \ge t(\eps_2)$ such that
\begin{align}
\label{est-zprime-L2}
\gamma_z \| z^{\prime} \|_{L^2(\Omega)}^2 
\le  (1+\eps_2) \lambda_{\mbox{\rm\tiny GS}} \frac{a_z(z^{\prime},z^{\prime})}{\mu_2} = 2(1+\eps_2) \frac{\lambda_{\mbox{\rm\tiny GS}}}{\mu_2} f.
\end{align}
Here we used that $\gamma_z \rightarrow \lambda_{\mbox{\rm\tiny GS}}$. Combining \eqref{f-tilde-prime-2}, \eqref{f-asymptotic-delta}, \eqref{crdest} and  \eqref{est-zprime-L2} we have
\begin{align*}
f^{\prime}(t) &\le - 2 f(t)
+
2(1+\eps_2) \frac{\lambda_{\mbox{\rm\tiny GS}}}{\mu_2} f(t) + \eps_1  \lambda_{\mbox{\rm\tiny GS}} \hspace{2pt} f(t)
+ C(\Omega,V,\beta,z_0) \eps_0 f(t).
\end{align*}
Selecting $\eps_0$, $\eps_1$ and $\eps_3$ sufficiently small and the corresponding times sufficiently large, we see that for any $\delta>0$ there exists a finite time $t_{\delta}$ such that for all $t\ge t_{\delta}$
\begin{align*}
f^{\prime}(t) &\le 2( 1 - \frac{\lambda_{\mbox{\rm\tiny GS}}}{\mu_2} - \delta ) f(t).
\end{align*}
By Gr\"onwall's lemma we obtain
$$
f(t) \le f( t_{\delta} ) e^{2 (1 - \frac{\lambda_{\mbox{\rm\tiny GS}}}{\mu_2} - \delta )(t-t_{\delta})} = f(t_{\delta}) e^{t_{\delta}} e^{-2t (1 - \frac{\lambda_{\mbox{\rm\tiny GS}}}{\mu_2} - \delta )}.
$$
Hence, for every $\delta>0$ there exists a constant $c_{\delta}$ and a finite time $t_{\delta}$, such that 
$$
f(t) \le c_{\delta}  e^{-2t (1 - \delta - \lambda_{\mbox{\rm\tiny GS}}/\mu_2 )} \qquad \mbox{for all } t \ge t_{\delta}. 
$$
Finally, we obtain for $0<\delta<1$ that
\begin{eqnarray*}
\lefteqn{a_0 ( z_{\mbox{\rm\tiny GS}}  - z(t) , z_{\mbox{\rm\tiny GS}}  - z(t) )}\\
&=& a_0\left( \int_{t}^{\infty} z^{\prime}(s)\hspace{2pt} \ds  , \int_{t}^{\infty} z^{\prime}(s) \ds  \right)
\le \left( \int_{t}^{\infty} a_0(z^{\prime}(s),z^{\prime}(s))^{1/2} \right)^2 \ds  \\
&\le& 2 \left( \int_{t}^{\infty} f(s)^{1/2} \ds \right)^2 
\le 2 c_{\delta} \left( \int_{t}^{\infty} e^{-s (1 - \delta - \lambda_{\mbox{\rm\tiny GS}}/\mu_2  ) } \ds \right)^2  \\
&\le& 8 c_{\delta} \hspace{2pt} e^{-2t (1 - \delta - \lambda_{\mbox{\rm\tiny GS}}/\mu_2  ) }.
\end{eqnarray*}
This finishes the proof.
\end{proof}

\section{Discrete Projected $a_z$-Sobolev Gradient Flow}
\label{sec:discrete:grad:flow}

In this section we propose and analyze a forward Euler discretization of the projected $a_z$-Sobolev gradient flow from Definition \ref{def-weighted-gradient-flow}. For this purpose, let $\{ \tau_n \}_{n\in \mathbb{N}}$ be a sequence of positive time steps that is bounded from above and below by
$$
0 < \tau_{\mbox{\tiny\rm min}} \le \tau_n \le \tau_{\mbox{\tiny\rm max}} < \infty.
$$
The time steps can be seen as parameters that should be selected sufficiently large for the sake of computational efficiency. In the following, we use the simplifying notation and write 
$$
\mathcal{G}^n z^n:= \mathcal{G}_{z^n} z^n
\qquad
\mbox{and}
\qquad
\gamma^n := \frac{(z^n, z^n)_{L^2(\Omega)}}{a_{z^n}( \mathcal{G}^n z^n ,\mathcal{G}^n z^n) }. 
$$
With this, we consider the following forward Euler discretization of the continuous $a_z$-gradient flow.
\begin{definition}[Method: GF$a_z$]
\label{definition-GF-az}
Let $z^0 \in H^1_0(\Omega)$ be given with $\| z^0 \|_{L^2(\Omega)}=1$. Then for $n\ge0$ the GF$a_z$-iteration $z^{n+1} \in H^1_0(\Omega)$ is defined as 
\begin{align}
\label{GFaz-eqn}
\hat{z}^{n+1} = (1 - \tau_n) z^n + \tau_n \gamma^n \mathcal{G}^n z^n 
\qquad\mbox{and}
\qquad
z^{n+1} = \frac{\hat{z}^{n+1}}{\| \hat{z}^{n+1} \|_{L^2(\Omega)}}.
\end{align}
\end{definition}
Since $z^n \in H^1_0(\Omega)$ and $\mathcal{G}^n z^n \in H^1_0(\Omega)$, the iterates are well-defined.

\begin{remark}[Nonlinear inverse iteration]
For the particular choice $\tau_n=1$ the iteration can be rewritten as
\begin{align*}
z^{n+1} = \frac{\mathcal{G}^n z^n }{\| \mathcal{G}^n z^n \|_{L^2(\Omega)}},
\end{align*}
which is the simplest form of the nonlinear inverse iteration (inverse power method). In this sense, GF$a_z$ is a generalized inverse iteration.
\end{remark}
We emphasize that a (near) optimal $\tau_n$ can be cheaply computed by (nearly) minimizing the energy $E(z^{n+1})$ as a function of $\tau_n$ along the given search direction as described in the following remark.
\begin{remark}[Adaptive GF$a_z$]
\label{adaptive-GFaz}
The proposed method GF$a_z$ can be easily combined with an adaptive step size control to compute optimal values for $\tau_n$ in each step. This involves the minimization of the function
$$f(\tau_n):=  E\left( \frac{(1 - \tau_n) z^n + \tau_n \gamma^n \mathcal{G}^n z^n} {\| (1 - \tau_n) z^n + \tau_n \gamma^n \mathcal{G}^n z^n \|_{L^2(\Omega)}} \right)
$$
w.r.t. $\tau_n\in (0,2)$. This can be done efficiently. Let us define
\begin{align*}
\alpha_0 &:= \int_{\Omega} |\nabla z^n|^2 + V |z^n|^2\dx,\quad \alpha_1 := 2 \gamma^n \int_{\Omega} \nabla z^n \cdot \nabla \mathcal{G}^n z^n + V \hspace{2pt} z^n \hspace{2pt} \mathcal{G}^n z^n\dx,\\
\alpha_2 &:= |\gamma^n|^2 \int_{\Omega} |\nabla \mathcal{G}^n z^n|^2 + V |\mathcal{G}^n z^n|^2\dx,
\end{align*}
and
\begin{align*}
\beta_0 &:= \frac{\beta}{2} \int_{\Omega} |z^n|^4\dx, \quad \beta_1 := 2 \beta \int_{\Omega}  (z^n)^3  \gamma^n \mathcal{G}^n z^n\dx ,\\ 
 \beta_2 &:= 3 \beta \int_{\Omega} |z^n|^2 |\gamma^n|^2 |\mathcal{G}^n z^n |^2\dx,\\
\beta_3 &:= \beta 2 \int_{\Omega}  z^n (\gamma^n)^3 (\mathcal{G}^n z^n )^3\dx, \quad \beta_4 :=  \frac{\beta}{2} \int_{\Omega} (\gamma^n)^4 |\mathcal{G}^n z^n |^4\dx
\end{align*}
and also
\begin{align*}
\zeta_0 &:= \int_{\Omega} |z^n|^2\dx,  \quad \zeta_1 := 2 \gamma^n \int_{\Omega}  z^n \hspace{2pt} \mathcal{G}^n z^n\dx \quad \mbox{and} \quad \zeta_2 := |\gamma^n|^2 \int_{\Omega} |\mathcal{G}^n z^n|^2\dx.
\end{align*}
The terms $\alpha_i$, $\beta_i$ and $\zeta_i$ have to be precomputed only once per time step (with a single grid walk). With these terms and the function
$$
s^n(\tau_n) := \Biggl( \sum_{i,j\ge 0:\;i+j=2} (1-\tau_n)^i \tau_n^j \zeta_{j} \Biggr)^{-1/2}
$$
we can see that $f(\tau_n)$ is given by 
\begin{align*}
f(\tau_n) =
\frac{1}{2} \Biggl( \sum_{i,j\ge 0:\;i+j=2} |s^n(\tau_n)|^2 (1-\tau_n)^i \tau_n^j \alpha_{j}
+ \sum_{i,j\ge 0:\;i+j=4} |s^n(\tau_n)|^4 (1-\tau_n)^i \tau_n^j \beta_{j} \Biggr).
\end{align*}
This quantity can be evaluated cheaply once that $\alpha_i$, $\beta_i$ and $\zeta_i$ were precomputed. The minimization of $f(\tau_n)$ on $(0,2)$ using e.g. golden section search leads to the (approximate) minimum $f(\tau_n^{\ast})$. The energy of $z^{n+1}$ is then given by $f(\tau_n^{\ast})$. Note that even without adaptivity,  the quantity $f(\tau_n)$ has to be computed, which is of the same order of complexity as the preprocessing step in the adaptive version. Hence, the computational overhead for using adaptivity is negligible. In particular, no additional linear system needs to be solved when using adaptivity with GF$a_z$. In contrast we observe that GF$L^2$ can typically not be efficiently combined with adaptivity, however, this also not as crucial as for the other methods since any sufficiently large choice for $\tau_n=\tau$ yields automatically a nearly optimal number of iterations for GF$L^2$ in terms of $\tau$.
\end{remark}
The remaining parts of this section are devoted to the numerical analysis of this scheme. 

\subsection{Intermediate mass growth} 
While conservation of mass is guaranteed by normalization in each step of the iteration, it is worth studying the change of mass that is associated with the map $z^n\mapsto \hat{z}^{n+1}$. It will turn out that mass cannot be diminished under this operation. 
Multiplying equation \eqref{GFaz-eqn} with $z^n$ and integrating over $\Omega$ yields
\begin{align}
\label{mass-equality-1}
\frac{1}{\tau_n} ( \hat{z}^{n+1} - z^n, z^n)_{L^2(\Omega)} = - ( z^n , z^n )_{L^2(\Omega)} + \frac{(z^n, z^n)_{L^2(\Omega)}}{(z^n,\mathcal{G}^n z^n)_{L^2(\Omega)} } (\mathcal{G}^n z^n, z^n )_{L^2(\Omega)} = 0.
\end{align}
This implies 
\begin{align*}
(\hat{z}^{n+1} - z^n, \hat{z}^{n+1} - z^n)_{L^2(\Omega)} 
&= (\hat{z}^{n+1} , \hat{z}^{n+1} )_{L^2(\Omega)} - 2 (\hat{z}^{n+1} , z^{n} )_{L^2(\Omega)} + (z^{n} , z^{n} )_{L^2(\Omega)} \\
&= (\hat{z}^{n+1} , \hat{z}^{n+1} )_{L^2(\Omega)} - 2 (z^{n} , z^{n} )_{L^2(\Omega)} + (z^{n} , z^{n} )_{L^2(\Omega)}.
\end{align*}
Hence $1=\| z^{n} \|_{L^2(\Omega)}^2 \le  \| z^{n} \|_{L^2(\Omega)}^2 + \| \hat{z}^{n+1} - z^n\|_{L^2(\Omega)}^2  = \| \hat{z}^{n+1} \|_{L^2(\Omega)}^2$. We summarize it:
\begin{lemma}[Intermediate mass growth]
For all $n$ it holds $1 = \| z^n \|_{L^2(\Omega)} \le  \| \hat{z}^{n+1} \|_{L^2(\Omega)}$.
Furthermore, the normalization error can be expressed as
\begin{align*}
 \| \hat{z}^{n+1} \|_{L^2(\Omega)} - 1 
= \frac{( z^n - z^{n+1} , z^n)_{L^2(\Omega)}}{ ( z^{n+1} , z^n)_{L^2(\Omega)}}.
\end{align*}
\end{lemma}
The previous lemma implies that the normalization of $\hat{z}^{n}$ to unit mass necessarily decreases the energy. Moreover, if the mass is not increased, i.e. $\| \hat{z}^{n+1} \|_{L^2(\Omega)}= \| z^{n} \|_{L^2(\Omega)}$, then this implies that $z^n = \hat{z}^{n+1} = z^{n+1}$.

\subsection{Energy dissipation}
The proof of energy reduction is established in several steps. First, using the result from the previous subsection, applying the energy inner product $a_{z^n} (\cdot, \cdot)$ to \eqref{GFaz-eqn} and testing with $\hat{z}^{n+1} - z^n$ yields
\begin{align}
\label{pseudo-energy-identity}
\nonumber\frac{1}{\tau_n} a_{z^n} (\hat{z}^{n+1} - z^n , \hat{z}^{n+1} - z^n ) &= - a_{z^n} ( z^n , \hat{z}^{n+1} - z^n ) + \gamma^n a_{z^n} ( \mathcal{G}^n z^n , \hat{z}^{n+1} - z^n ) \\
\nonumber&= - a_{z^n} ( z^n , \hat{z}^{n+1} - z^n ) + \gamma^n ( z^n , \hat{z}^{n+1} - z^n )_{L^2(\Omega)} \\
&\overset{\eqref{mass-equality-1}}{=} - a_{z^n} ( z^n , \hat{z}^{n+1} - z^n ).
\end{align}
This leads to a preliminary lower bound for the energy difference.
\begin{lemma}[Sharp lower bounds for the energy difference]
\label{sharp-lower-bounds-lemma}
If $\tau_n\le 2/5$ then
\begin{equation}
\label{energy-difference}E(z^n)  - E(\hat{z}^{n+1} ) 
\ge- \int_{\Omega} \tfrac{\beta}{2} |\hat{z}^{n+1} - z^{n}|^4 \dx+ \biggl(\frac{1}{\tau_n} - \tfrac{1}{2}\biggr) a_0( \hat{z}^{n+1} - z^{n}  , \hat{z}^{n+1} - z^n ).
\end{equation}
For $\tau_n \ge 2$, either $E(\hat{z}^{n+1})>E(z^{n+1})$ or $z^n$ is  already a critical point.
\end{lemma}

\begin{proof}
Set $\tau:=\tau_n$. We get
\begin{eqnarray*}
\lefteqn{ a_{z^n}(\hat{z}^{n+1} - z^n , \hat{z}^{n+1} - z^n ) = a_{z^n}(\hat{z}^{n+1}, \hat{z}^{n+1} )  - 2 a_{z^n}(z^{n} , \hat{z}^{n+1}) + a_{z^n}(z^{n} , z^{n})  }\\
&=&
a_{z^n}(\hat{z}^{n+1}, \hat{z}^{n+1} ) 
- 2 a_{z^n}(z^{n} , \hat{z}^{n+1} - z^n )
- a_{z^n}(z^{n} , z^{n})  \\
&\overset{\eqref{pseudo-energy-identity}}{=}&
a_{z^n}(\hat{z}^{n+1}, \hat{z}^{n+1} ) 
+ \tfrac{2}{\tau} a_{z^n}( \hat{z}^{n+1} - z^{n}  , \hat{z}^{n+1} - z^n )
- a_{z^n}(z^{n} , z^{n}),
\end{eqnarray*}
which implies
\begin{eqnarray*}
a_{z^n}(z^{n} , z^{n})
&=&
a_{z^n}(\hat{z}^{n+1}, \hat{z}^{n+1} ) 
+ \bigl(\tfrac{2}{\tau} - 1\bigr) a_{z^n}( \hat{z}^{n+1} - z^{n}  , \hat{z}^{n+1} - z^n ).
\end{eqnarray*}
Observe that 
$$
a_{z^n}(z^{n} , z^{n})= 2 E(z^n) + \tfrac{\beta}{2} \int_{\Omega} |z^{n}|^4\dx
$$
and
\begin{align*}
a_{z^n}(\hat{z}^{n+1}, \hat{z}^{n+1} ) = 2 E(\hat{z}^{n+1}) + \tfrac{\beta}{2} \hspace{-2pt}\int_{\Omega} |\hat{z}^{n+1}|^2 ( |z^{n}|^2 \hspace{-2pt}-\hspace{-2pt}  |\hat{z}^{n+1}|^2) \dx + \tfrac{\beta}{2} \hspace{-2pt}\int_{\Omega} |z^{n}|^2 |\hat{z}^{n+1}|^2\dx.
\end{align*}
Combining everything yields
\begin{eqnarray*}
 E(z^n)  - E(\hat{z}^{n+1} ) 
&=&
- \tfrac{\beta}{4} \int_{\Omega} (|\hat{z}^{n+1}|^2 - |z^{n}|^2)^2 \dx
+ \bigl(\tfrac{1}{\tau} - \tfrac{1}{2}\bigr) a_{z^n}( \hat{z}^{n+1} - z^{n}  , \hat{z}^{n+1} - z^n ) \\
&=&
- \tfrac{\beta}{4} \int_{\Omega} (|\hat{z}^{n+1}|^2 - |z^{n}|^2)^2 \dx
+ \bigl(\tfrac{1}{\tau} - \tfrac{1}{2}\bigr) \beta \int_{\Omega} |z^n|^2 |\hat{z}^{n+1} - z^{n}|^2\dx\\
&\enspace& \qquad
+ \bigl(\tfrac{1}{\tau} - \tfrac{1}{2}\bigr) a_0( \hat{z}^{n+1} - z^{n}  , \hat{z}^{n+1} - z^n ) \\
&=&
- \tfrac{\beta}{4} \int_{\Omega} |\hat{z}^{n+1}+ z^{n}|^2  |\hat{z}^{n+1} - z^{n}|^2\dx 
+ \bigl(\tfrac{1}{\tau} - \tfrac{1}{2}\bigr) \beta \int_{\Omega} |z^n|^2 |\hat{z}^{n+1} - z^{n}|^2\dx\\
&\enspace& \qquad+ \bigl(\tfrac{1}{\tau} - \tfrac{1}{2}\bigr) a_0( \hat{z}^{n+1} - z^{n}  , \hat{z}^{n+1} - z^n ) \\
&=&
\beta \int_{\Omega}  |\hat{z}^{n+1} - z^{n}|^2  \left( - \tfrac{1}{4} |\hat{z}^{n+1}+ z^{n}|^2 +  \bigl(\tfrac{1}{\tau} - \tfrac{1}{2}\bigr) |z^n|^2 \right)\dx\\
&\enspace& \qquad+ \bigl(\tfrac{1}{\tau} - \tfrac{1}{2}\bigr) a_0( \hat{z}^{n+1} - z^{n}  , \hat{z}^{n+1} - z^n ).
\end{eqnarray*}
For $\tau\ge2$, the right hand side is negative which implies a growth of energy. It only remains to find a lower bound for the first term. Here we estimate
\begin{eqnarray*}
\int_{\Omega}  |\hat{z}^{n+1} - z^{n}|^2  |\hat{z}^{n+1}+ z^{n}|^2\dx
&\le& \int_{\Omega}   2 |\hat{z}^{n+1} - z^{n}|^4 + 
8 |\hat{z}^{n+1} - z^{n}|^2 |z^n|^2\dx,
\end{eqnarray*}
which yields
\begin{eqnarray*}
\lefteqn{\int_{\Omega}  |\hat{z}^{n+1} - z^{n}|^2  \left( - \tfrac{1}{4} |\hat{z}^{n+1}+ z^{n}|^2 +  \bigl(\tfrac{1}{\tau} - \tfrac{1}{2}\bigr) |z^n|^2 \right) \dx}\\
&\ge& - \int_{\Omega} \tfrac{1}{2} |\hat{z}^{n+1} - z^{n}|^4 \dx + \int_{\Omega} \bigl(\tfrac{1}{\tau} - \tfrac{5}{2}\bigr) |\hat{z}^{n+1} - z^{n}|^2 |z^n|^2 \dx
\ge - \int_{\Omega} \tfrac{1}{2} |\hat{z}^{n+1} - z^{n}|^4\dx,
\end{eqnarray*}
where we used $\tau\le 2/5$. This finishes the proof.
\end{proof}
\begin{remark}[Adaptive time steps]
	Observe that if the time steps $\tau_n$ are chosen adaptively, then asymptotically any choice $\tau_n < 2$ is admissible. This however requires that the previous time steps (with typically smaller step size) were such that the iterates $z^n$ are in a sufficiently small neighbourhood of a critical point. This is because in the convergent regime, the first term in \eqref{energy-difference} (which is of order four) is eventually negligible, compared to the dominant second term which is only of second order. We will not exploit this observation, but believe that it is worth mentioning.
\end{remark}

With Lemma \ref{sharp-lower-bounds-lemma} we can now prove energy reduction for sufficiently small time steps.
\begin{lemma}[Energy reduction]\label{lemma-energy-reduction-euler}
There exists $0<\tau_{\mbox{\rm\tiny max}}<2$ (which depends on $\beta$ and $E(z^0)$) such that for all $\tau_n\le \tau_{\mbox{\rm\tiny max}}$ 
$$
E( z^{n+1} ) \le E({\hat{z}^{n+1}}) \le E(z^{n}).
$$
\end{lemma}
\begin{remark}[Energy reduction for $\tau_n=1$]\label{remark-tau-1}
Numerically, we could observe the coupling between $\tau_{\mbox{\rm\tiny max}}$ and the energy of $z^0$ at several occasions, i.e. if $E(z_0)$ was large then the step size $\tau_n$ had to be reduced to obtain reduction of the energy. However, we never observed that $\tau_{\mbox{\rm\tiny max}}$ dropped below one. In this connection we shall note that, for $\tau_n=1$, the GF$a_{z}$ can be seen as a GF$L^2$ realization applied to the Schr\"odinger operator whose spectrum was shifted by $+1$. Consider the GPE with the modified potential $V_{\mbox{\tiny\rm mod}}:=V -1$. Applying GF$L^2$ to this modified problem gives the same iterations as applying the GF$a_{z}$ iterations to the GPE with original potential (for the particular choice $\tau_n=1$). Hence, both methods produce the same approximations $z^n$. Using the results obtained in \cite[Lemma 2.10]{BaD04} for GF$L^2$ with $\tau_n=1$ we can hence argue that the GF$a_{z}$ iterates are guaranteed to reduce a functional of the form $\tilde{E}(v)=E(v)+\tfrac{\beta}{4}\| v \|_{L^4(\Omega)}^4$. Since $\lambda = 2E(u) + \tfrac{\beta}{2}\| u \|_{L^4(\Omega)}^4=2\tilde{E}(v)$, this can be seen as minimizing an \quotes{eigenvalue functional} instead of the original energy functional.
\end{remark}
\begin{proof}[Proof of Lemma \ref{lemma-energy-reduction-euler} (by induction)]
{\it Step $n=0$}: With \eqref{pseudo-energy-identity} we have
\begin{align*}
a_0(\hat{z}^1 - z^0 , \hat{z}^1 - z^0 ) \le \tau_0^2 4 E(z^0).
\end{align*}
If $\tau_0^2 \le (4 E(z^0))^{-1}$, we have $a_0(\hat{z}^1 - z^0 , \hat{z}^1 - z^0 ) \le 1$ and hence
\begin{align*}
\int_{\Omega} \tfrac{\beta}{2} |\hat{z}^1 - z^{0}|^4\dx
\lesssim \beta a_0(\hat{z}^1 - z^0 , \hat{z}^1 - z^0 )^2
\le \beta a_0(\hat{z}^1 - z^0 , \hat{z}^1 - z^0 )
\end{align*}
Together with \eqref{energy-difference}, we conclude 
\begin{eqnarray*}
E(z^0)  - E(\hat{z}^1 ) \ge (\tfrac{1}{\tau_0} - \tfrac{1}{2} - C \beta) a_0( \hat{z}^1 - z^{0}  , \hat{z}^1 - z^0 ).
\end{eqnarray*}
Hence, there exists $\tau_{\mbox{\rm\tiny max}} \lesssim \min\{ \beta^{-1} , E(z^0)^{-1/2} \} $ such that for all $\tau_0\le \tau_{\mbox{\rm\tiny max}}$ we have
\begin{eqnarray*}
E(z^0)  - E(\hat{z}^1 )  \ge 0 \qquad \overset{\| \hat{z}^1\|_{L^2(\Omega)}\ge1 }{\Rightarrow}
\qquad  E(z^1 ) \le  E(\hat{z}^1 ) \le E(z^0).
\end{eqnarray*}
\noindent {\it Step $n\mapsto n+1$}: Let $E(z^{n} ) \le E(z^{0} )$ and $\tau_n\le \tau_{\mbox{\rm\tiny max}}$ with $\tau_{\mbox{\rm\tiny max}}$ as for $n=0$. Using \eqref{pseudo-energy-identity} and $E(z^{n} ) \le E(z^{0} )$
we have
\begin{align*}
a_0({\hat{z}^{n+1}} - z^n , {\hat{z}^{n+1}} - z^n ) \le \tau_n^2 4 E(z^n) \le \tau_n^2 4 E(z^0) \le 1.
\end{align*}
Analogously as for $n=0$, we have
\begin{align*}
\int_{\Omega} \tfrac{\beta}{2} |{\hat{z}^{n+1}} - z^{n}|^4\dx
\lesssim \beta \hspace{2pt} a_0({\hat{z}^{n+1}} - z^n , {\hat{z}^{n+1}} - z^n )
\end{align*}
and, hence,
\begin{eqnarray}
\label{Energy-diff-by-H1-diff}
E(z^n)  - E({\hat{z}^{n+1}} )  \ge c_{\tau} a_0( {\hat{z}^{n+1}} - z^{n}  , {\hat{z}^{n+1}} - z^{n} ) \ge 0.
\end{eqnarray}
Note that $c_{\tau} \rightarrow \infty$ for $\tau \rightarrow 0$.
\end{proof}

\subsection{Global convergence}
We have the following main result on the global convergence of the discrete gradient flow.
\begin{theorem}\label{lemma-weak-convergence-Euler}
We consider the GF$a_z$-approach stated in Definition \ref{definition-GF-az}.
Assume that the time steps fulfill $\tau_n\le \tau_{\mbox{\rm\tiny max}}$ as in Lemma \ref {lemma-energy-reduction-euler} and that they are non-degenerate in the sense that $\tau_n \ge \tau_{\mbox{\rm\tiny min}} > 0$. Then there exists a limit energy 
$E^{\ast}:=\lim_{n\rightarrow \infty} E(z^n)$. 
Furthermore, there exists a subsequence $\{ z^{n_i} \}_{i \in \mathbb{N}}$ of $\{ z^n \}_{n \in \mathbb{N}}$, such that 
$z^{n_i}\rightarrow z^{\ast}$ strongly 
in $H^1_0(\Omega)$ to some limit $z^{\ast}\in H^1_0(\Omega)$ with $\| z^{\ast} \|_{L^2(\Omega)} =1$ and $E(z^{\ast}) = E^{\ast}$.
With 
$$\lambda^{\ast}:=
\| \mathcal{G}^{\ast} z^{\ast} \|_{L^2(\Omega)}^{-1}
=\lim_{i\rightarrow\infty} \gamma^{n_i},$$
we have that $z^{\ast}$ is an eigenfunction to the Gross-Pitaevskii equation and fulfills 
\begin{align*}
a_{z^{\ast}} (z^{\ast} , v ) = \lambda^{\ast} (z^{\ast} , v )_{L^2(\Omega)} \qquad \mbox{for all } v \in H^1_0(\Omega).
\end{align*}
Any other convergent subsequence of $\{ z^n \}_{n \in \mathbb{N}}$ will also converge strongly in $H^1(\Omega)$ to an $L^2$-normalized eigenfunction of the GPE with energy level $E^{\ast}$. However, the corresponding eigenvalue might be different from $\lambda^{\ast}$ above.
\end{theorem} 
\begin{remark}\label{rem:linear}
Note that 
$\mathcal{G}^{\ast}(z^{\ast}) = \frac{1}{\lambda^{\ast}}z^{\ast}$. 
\end{remark}
Recall that the limit energy in Theorem \ref{lemma-weak-convergence-Euler} depends crucially on $z^0$, but potentially it can also depend on the choice of the sequence $\{\tau_n\}_{n \in \mathbb{N}}$. If there exists only one eigenfunction (up to normalization and multiplication with $-1$) for the energy level $E^{\ast}$, then we have convergence of the full sequence in Theorem \ref{lemma-weak-convergence-Euler}.
\begin{proof}[Proof of Theorem~\ref{lemma-weak-convergence-Euler}]
As $E(z^n)$ is a monotonically decreasing sequence the limit $E^{\ast}:=\lim_{n\rightarrow \infty} E(z^n)$ exists. This means that $\{z^n\}_{n\in \mathbb{N}}$ is a bounded sequence in $H^1_0(\Omega)$ from which we can extract a subsequence, for brevity still denoted by $\{z^n\}_{n\in \mathbb{N}}$, that converges weakly in $H^1_0(\Omega)$ to some limit function $z^{\ast} \in H^1_0(\Omega)$ with $\| z^{\ast} \|_{L^2(\Omega)}=1$. For space dimension $d\le 3$ the Rellich-Kondrachov theorem guarantees that $z^n$ converges to $z^{\ast}$, strongly in $L^p(\Omega)$ for $p<6$. First, we note that \eqref{Energy-diff-by-H1-diff} implies that $\| \hat{z}^{n+1} - z^n\|_{L^2(\Omega)} \rightarrow 0$ and consequently
\begin{align*}
\| (1-\tau_n)z^n + \tau_n \gamma^n \mathcal{G}^n z^n \|_{L^2(\Omega)} =
\| \hat{z}^{n+1} \|_{L^2(\Omega)} \overset{n \rightarrow \infty}{\longrightarrow} 1.
\end{align*}
Furthermore, it holds for any $v \in H^1_0(\Omega)$
\begin{align*}
a_{z^n} ( \mathcal{G}^n z^n ,v ) = ( z^{n} , v )_{L^2(\Omega)}  \overset{n \rightarrow \infty}{\longrightarrow} ( z^{\ast} , v )_{L^2(\Omega)} = 
a_{z^{\ast}} ( \mathcal{G}^{\ast} z^{\ast} ,v ).
\end{align*}
Using the aforementioned Rellich-Kondrachov embedding we have that $|z^{n}|^2 \rightarrow |z^{\ast}|^2$ strongly in $L^2(\Omega)$ and hence
\begin{align*}
a_{z^{\ast}} ( \mathcal{G}^n z^n ,v ) \overset{n \rightarrow \infty}{\longrightarrow} 
a_{z^{\ast}} ( \mathcal{G}^{\ast} z^{\ast} ,v ).
\end{align*}
The later equation implies that $\mathcal{G}^n z^n$ converges weakly in $H^1_0(\Omega)$ (and strongly in $L^2(\Omega)$) to $ \mathcal{G}^{\ast} z^{\ast}$.
Combining the strong $L^2$-convergence of $\mathcal{G}^n z^n$ and $z^n$ we obtain
\begin{align*}
(\gamma^n)^{-1} =
a_{z^n}( \mathcal{G}^n z^n ,\mathcal{G}^n z^n) =
 ( \mathcal{G}^n z^n , z^n)_{L^2(\Omega)} \overset{n \rightarrow \infty}{\longrightarrow}   ( \mathcal{G}^{\ast} z^{\ast} ,z^{\ast} )_{L^2(\Omega)} =: (\gamma^{\ast})^{-1}.
\end{align*}
Combining all the results we can use
$$
z^{n+1} = \frac{(1-\tau_n)z^n + \tau_n \gamma^n \mathcal{G}^n z^n}{\| (1-\tau_n)z^n + \tau_n \gamma^n \mathcal{G}^n z^n \|_{L^2(\Omega)}}
$$
and pass to the limit for any $v\in H^1_0(\Omega)$ in
\begin{align*}
0 \longleftarrow \hspace{3pt}&\tau_n^{-1} a_{z^{\ast}}( \| (1-\tau_n)z^n + \tau_n \gamma^n \mathcal{G}^n z^n \|_{L^2(\Omega)} z^{n+1} - z^n , v ) \\
&=  a_{z^{\ast}}( - z^{n} + \gamma^n \mathcal{G}^n z^n, v ) \longrightarrow - a_{z^{\ast}}( z^{\ast} , v ) + \gamma^{\ast}  a_{z^{\ast}}( \mathcal{G}^{\ast} z^{\ast}, v ).
\end{align*}
Thus, $a_{z^{\ast}}( z^{\ast} , v ) = \gamma^{\ast} ( z^{\ast} , v)_{L^2(\Omega)}$ for all $v\in H^1_0(\Omega)$.
To verify the convergence of the energy, i.e. $E^{\ast}=E(z^{\ast})$, observe that
$$
\gamma_n = \tau_n^{-1} a_{z^n}(\hat{z}^{n+1} , z^{n} ) - \tfrac{1 - \tau_n}{\tau_n} a_{z^n}( z^n , z^n).
$$
Using this expression, we have
\begin{eqnarray*}
\lefteqn{2 \hspace{2pt} | E(z^{\ast}) - E(z^n)| = \left| \gamma^{\ast} - \tfrac{\beta}{2} \int_{\Omega} |z^{\ast}|^4 \dx
- a_{z^n}(z^n,z^n) + \tfrac{\beta}{2} \int_{\Omega} |z^{n}|^4 \dx\right|} \\
&=& \left| \gamma^{\ast} - \gamma_n + \gamma_n - \tfrac{\beta}{2} \int_{\Omega} |z^{\ast}|^4 \dx
- a_{z^n}(z^n,z^n) + \tfrac{\beta}{2} \int_{\Omega} |z^{n}|^4 \dx\right| \\
&\le&  | \gamma^{\ast} - \gamma_n| + \tfrac{\beta}{2} \int_{\Omega} \bigl| |z^{n}|^4 - |z^{\ast}|^4  \bigr|\dx \\
&\enspace&\qquad + | \tau_n^{-1} a_{z^n}(\hat{z}^{n+1} , z^{n} ) - \tfrac{1 - \tau_n}{\tau_n} a_{z^n}( z^n , z^n) - a_{z^n}(z^n,z^n) | \\
&=& | \gamma^{\ast} - \gamma_n| + \tfrac{\beta}{2} \int_{\Omega}\bigl| |z^{n}|^4 - |z^{\ast}|^4  \bigr|\dx  + \tau_n^{-1} |  a_{z^n}(\hat{z}^{n+1} - z^n , z^{n} )|\\
&\overset{\eqref{Energy-diff-by-H1-diff}}{\le}&  | \gamma^{\ast} - \gamma_n| + \tfrac{\beta}{2} \int_{\Omega} \bigl| |z^{n}|^4 - |z^{\ast}|^4  \bigr|\dx + C(\tau_{\mbox{\tiny\rm min}}, \tau_{\mbox{\tiny\rm max}},z^0) \sqrt{ E( \hat{z}^{n+1}) - E(z^n) }.
\end{eqnarray*}
For all terms on the right-hand side we verified (strong) convergence. Consequently we have 
$$
| E(z^{\ast}) - E^{\ast}|=\lim_{n\rightarrow \infty}| E(z^{\ast}) - E(z^n)|=0.
$$
The strong convergence of $z^n$ in $H^1(\Omega)$ follows readily from the previous result as it implies $\lim_{n\rightarrow \infty} \| z^n \|_{H^1(\Omega)} =  \| z^{\ast} \|_{H^1(\Omega)}$.
\end{proof}
It is easily seen that all proofs in this section remain valid, if we replace the space $H^1_0(\Omega)$ in the GF$a_z$-approach by a finite dimensional subspace, e.g. in a spatial finite element discretization.
\begin{corollary}[Convergence of the fully discrete GF$a_z$]
Let $V_h \subset H^1_0(\Omega)$ be a finite dimensional subspace and let $\mathcal{G}_{z}^h(z_h) \in V_h$ solve
$a_z( \mathcal{G}_{z}^h(z_h) , v_h ) = ( z_h , v_h )_{L^2(\Omega)}$ for all $v_h \in V_h$. For $z_{h}^0 \in V_h$ with $\| z_{h}^0\|_{L^2(\Omega)}=1$ we consider the GF$a_z$ iteration
\begin{align*}
\hat{z}_h^{n+1} = (1 - \tau) z_h^n + \tau \hspace{2pt} (z^n_h,\mathcal{G}_{z^n_h}^{h} z^n_h)_{L^2(\Omega)}^{-1} \hspace{2pt} \mathcal{G}_{z^n_h}^{h} z_h^n 
\qquad\mbox{and}
\qquad
z_h^{n+1} = \frac{\hat{z}_h^{n+1}}{\| \hat{z}_h^{n+1} \|_{L^2(\Omega)}}.
\end{align*}
If $\tau \le \tau_{\mbox{\rm\tiny max}}$ then the energy is strictly reduced and there exists a limit energy 
$E^{\ast}_h:=\lim_{n\rightarrow \infty} E(z^n_h)$. Furthermore, up to subsequences, we have $z^{n}_h\rightarrow z^{\ast}_h$ strongly in $H^1_0(\Omega)$ where $z^{\ast}_h\in V_h$
with $\| z^{\ast}_h \|_{L^2(\Omega)} =1$ and $E(z^{\ast}_h) = E^{\ast}_h$ is a discrete eigenfunction of the GPE, i.e. there is $\lambda_h^{\ast}$ so that
\begin{align*}
a_{z^{\ast}_h} (z^{\ast}_h , v_h ) = \lambda^{\ast}_h (z^{\ast}_h , v_h )_{L^2(\Omega)} \qquad \mbox{for all } v_h \in V_h.
\end{align*}
\end{corollary}
The convergence of approximate eigenpairs $(\lambda^{\ast}_h,z^{\ast}_h)$ to the true ones has been investigated and analyzed in \cite{CCM10,HMP14b,CCH18}.

\section{Global convergence to the ground state}
\label{sec:positive:eigenstates}
Theorem \ref{lemma-weak-convergence-Euler} shows uniqueness of the limit of the discrete flow $(z^n)$ under uniqueness of the  eigenfunction (up to normalization) on the energy level $E^{\ast}$. The latter assumption can be relaxed in the  particular case of eigenstates that are strictly positive in the interior of $\Omega$. As we have already seen, there exists at least one such state, which is the ground state of the energy functional $E$. In this section we will prove that it is also the only one. This is crucial for the following main result.
\begin{theorem}\label{global-convergence-ground state}
Consider the GF$a_z$-approach. Let the assumptions of Theorem \ref{lemma-weak-convergence-Euler} hold and $\tau_{n}\le1$ for all $n$. Then for any starting value
$z^0 \in H^1_0(\Omega)$ with $\| z^0 \|_{L^2(\Omega)}=1$ and $z^0 \ge 0$ the (full) sequence $(z^n)$ converges strongly in $H^1(\Omega)$ to 
the positive ground state $z_{\mbox{\rm\tiny GS}}$ (which is unique according to Proposition \ref{basic-existence-ground-state}).
\end{theorem}
\begin{remark}\label{how-to-chose-starting-value}
Theorem \ref{global-convergence-ground state} guarantees global convergence to the ground state, provided that the starting value $z^0$ is not changing its sign. Additionally, starting from a non-negative $z^0$, the gradient flow does not converge to an excited state. A sign-changing starting value is compulsory for the computation of excited states.
\end{remark}
Before we can prove Theorem \ref{global-convergence-ground state}, a few auxiliary results are required. The first result relates the positive eigenfunctions in the spectrum of the GPE to the ground states of a linear operator obtained by freezing the density. The lemma can be proved analogously to a similar result obtained in \cite[Lemma 2]{CCM10}.

\begin{lemma}\label{relation-ground state-linear-problem}
Let $z^{\ast} \in H^1_0(\Omega)$ with $\| z^{\ast} \|_{L^2(\Omega)}=1$ be an eigenstate of the GPE with eigenvalue $\lambda^{\ast}>0$, i.e.
\begin{align*}
\langle E^{\prime}(z^{\ast}) , v \rangle = \lambda^{\ast} (z^{\ast} , v)_{L^2(\Omega)} \qquad \mbox{for all } v \in H^1_0(\Omega).
\end{align*}
If $z^{\ast}\ge 0$ in $\Omega$, then $z^{\ast}$ can be characterized as the as the unique positive ($L^2$-normalized) ground state to the {\rm linear} operator $\mathcal{G}_{z^\ast}^{-1}$ (see Remark~\ref{rem:linear}) and $z^{\ast}$ must be even strictly positive in the interior of $\Omega$.
\end{lemma}
Next, we prove that the positive eigenstate is unique and hence always the ground state.
\begin{lemma}[Uniqueness of positive eigenstates]
\label{lemma-uniqueness-positive-critical-points}
There is a unique positive eigenfunction to the GPE \eqref{weak-problem}, which is the ground state.
\end{lemma}
\begin{proof}
We recall the Picone identity (cf. \cite{Pic1910,BrF14}), which implies that for two functions $u,v\in H^1(\Omega)$ with $u\ge0$ and $v>0$ in $\Omega$ it holds
\begin{align}
\label{Picone}
\int_{\Omega} \nabla v \cdot \nabla \left( \frac{u^2}{v} \right)\dx \le \int_{\Omega} |\nabla u|^2\dx.
\end{align} 
From Lemma \ref{relation-ground state-linear-problem} we know that any nonnegative eigenfunction must be even strictly positive. Let us therefore assume we have two positive $L^2$-normalized eigenfunctions $z_{\mbox{\rm\tiny GS}},z_{\mbox{\rm\tiny ES}} \in H^1_0(\Omega)$ to the Gross-Pitaevskii equation, where $z_{\mbox{\rm\tiny GS}}$ is the unique ground state with minimal energy $E_{\mbox{\rm\tiny GS}}$ and eigenvalue $\lambda_{\mbox{\rm\tiny GS}}$ and $z_{\mbox{\rm\tiny ES}}$ is an excited state with energy $E_{\mbox{\rm\tiny ES}}> E_{\mbox{\rm\tiny GS}}$. Using $\| z_{\mbox{\rm\tiny GS}} \|_{L^2(\Omega)}=1$ it holds
\begin{eqnarray*}
\lefteqn{\lambda_{\mbox{\rm\tiny ES}}
= \lambda_{\mbox{\rm\tiny ES}} ( z_{\mbox{\rm\tiny ES}} , \frac{z^2_{\mbox{\rm\tiny GS}}}{ z_{\mbox{\rm\tiny ES}}} )_{L^2(\Omega)} 
= a_{ z_{\mbox{\rm\tiny ES}} }( z_{\mbox{\rm\tiny ES}}  , \frac{z^2_{\mbox{\rm\tiny GS}}}{ z_{\mbox{\rm\tiny ES}}}  ) }\\
&\overset{\eqref{Picone}}{\le}&
\int_{\Omega} |  \nabla z_{\mbox{\rm\tiny GS}}  |^2 \dx
+ \int_{\Omega} V | z_{\mbox{\rm\tiny GS}}  |^2 \dx
+ \beta \int_{\Omega} | z_{\mbox{\rm\tiny ES}}  |^2 | z_{\mbox{\rm\tiny GS}}  |^2\dx \\
&\le& \int_{\Omega} |  \nabla z_{\mbox{\rm\tiny GS}}  |^2 \dx
+ \int_{\Omega} V | z_{\mbox{\rm\tiny GS}}  |^2\dx + \beta \int_{\Omega} | z_{\mbox{\rm\tiny GS}}  |^4 \dx
 - \tfrac{\beta}{2} \int_{\Omega} | z_{\mbox{\rm\tiny GS}}  |^4 \dx
 + \tfrac{\beta}{2} \int_{\Omega} | z_{\mbox{\rm\tiny ES}}  |^4 \dx \\
 &=& \lambda_{\mbox{\rm\tiny GS}}  - \tfrac{\beta}{2} \int_{\Omega} | z_{\mbox{\rm\tiny GS}}  |^4 \dx
 + \tfrac{\beta}{2} \int_{\Omega} | z_{\mbox{\rm\tiny ES}}  |^4\dx.
\end{eqnarray*}
%
%
We conclude that
\begin{align*}
2 E_{\mbox{\rm\tiny ES}} = 
\lambda_{\mbox{\rm\tiny ES}} - \tfrac{\beta}{2} \int_{\Omega} | z_{\mbox{\rm\tiny ES}}  |^4\dx
\le  \lambda_{\mbox{\rm\tiny GS}}  - \tfrac{\beta}{2} \int_{\Omega} | z_{\mbox{\rm\tiny GS}}  |^4 \dx
= 2 E_{\mbox{\rm\tiny GS}}.
\end{align*}
This is a contradiction to the assumption that $z_{\mbox{\rm\tiny ES}}$ was an excited state with $E_{\mbox{\rm\tiny ES}}> E_{\mbox{\rm\tiny GS}}$. Hence, we have $z_{\mbox{\rm\tiny ES}}=z_{\mbox{\rm\tiny GS}}$ which is unique.
\end{proof}

The next (fairly obvious) result shows that positivity is preserved by the iteration.
\begin{lemma}
Let $v, z\in H^1_0(\Omega)$. Then
\begin{align*}
v\ge 0 \qquad \Rightarrow \qquad \mathcal{G}_{z}(v) \ge 0.
\end{align*}
In particular, if $z^n\ge 0$ and $\tau_{n}\le1$ then $z^{n+1}\ge0$.
\end{lemma}
\begin{proof}
We can characterize $\mathcal{G}_{z}(v)$ as the unique minimizer of
\begin{align*}
F(w) := a_{z}(w,w) - \tfrac{1}{2} (v, w)_{L^2(\Omega)}
\end{align*} 
among all $w\in H^1_0(\Omega)$. However, since it holds
$F( |\mathcal{G}_{z}(v)| ) \le  F( \mathcal{G}_{z}(v) )$
we conclude by uniqueness $|\mathcal{G}_{z}(v)| = \mathcal{G}_{z}(v)$, which guarantees that $ \mathcal{G}_{z}(v)$ cannot become negative. The positivity of $z^{n+1}$ follows immediately with $\hat{z}^{n+1} = (1-\tau_n)z^n + \tau_n \gamma^n \mathcal{G}^n z^n$, where $\gamma^n,z^n,\mathcal{G}^n z^n \ge 0$ and $\tau_n\le 1$.
\end{proof}
We are now ready to prove the main result of this section.
\begin{proof}[Proof of Theorem \ref{global-convergence-ground state}]
Let $z^{\ast} \in H^1_0(\Omega)$ be any of the limits of a subsequence of $z^n$ whose existence is guaranteed by Theorem \ref{lemma-weak-convergence-Euler} with $E^{\ast}=E(z^{\ast})$. Then we have that for any $z\in H^1_0(\Omega)$ with $\| z\|_{L^2(\Omega)}=\| z^{\ast}\|_{L^2(\Omega)}=1$ it holds
\begin{eqnarray}
\nonumber \lefteqn{E(z) - E(z^{\ast})} \\ 
\nonumber &=&
\tfrac{1}{2} a_0(z,z) + \tfrac{\beta}{4} \int_{\Omega} |z|^4 \dx- \tfrac{1}{2} a_{z^{\ast}}(z^{\ast},z^{\ast}) + \tfrac{\beta}{4} \int_{\Omega} |z^{\ast}|^4  \dx\\
\nonumber  &=&
\tfrac{1}{2} \left( a_{z^{\ast}}(z,z) - a_{z^{\ast}}(z^{\ast},z^{\ast})  \right) +
 \tfrac{\beta}{2} \int_{\Omega} |z|^4\dx
+ \tfrac{\beta}{4} \int_{\Omega} |z^{\ast}|^4\dx  - \tfrac{\beta}{2} \int_{\Omega} |z|^2 |z^{\ast}|^2 \dx\\
 \nonumber&=&
\tfrac{1}{2} \left( a_{z^{\ast}}(z - z^{\ast} ,z -z^{\ast}) - \lambda^{\ast} (z - z^{\ast} ,z -z^{\ast})_{L^2(\Omega)} \right) + \tfrac{\beta}{4} \int_{\Omega} (|z|^2  - |z^{\ast}|^2)^2\dx\\[-0.7em]
\label{energy-identity-z-ast}
\end{eqnarray}
Here $\lambda^{\ast}$ is the eigenvalue to the eigenfunction $z^{\ast}$. Since $z^{\ast}$ is the strong $H^1$-limit of a sequence of positive functions $z^{n_i}$, pointwise convergence almost everywhere ensures that $z^{\ast}\ge0$. Hence, we can apply Lemma \ref{relation-ground state-linear-problem} that guarantees $z^{\ast}>0$ and that $\lambda^{\ast}>0$ is the ground state eigenvalue of the linear operator $\mathcal{G}_{z^\ast}^{-1}$. Hence, it holds $\langle \mathcal{G}_{z^\ast}^{-1} v , v \rangle \ge \lambda^{\ast} (v,v)_{L^2(\Omega)}$ for any $v\in H^1_0(\Omega)$ or respectively 
$$
a_{z^{\ast}} (v,v ) - \lambda^{\ast} (v,v)_{L^2(\Omega)} = \langle \mathcal{G}_{z^\ast}^{-1} v , v \rangle - \lambda^{\ast} (v,v)_{L^2(\Omega)} \ge 0.
$$
Using this finding in \eqref{energy-identity-z-ast} implies
\begin{align}
\int_{\Omega} (|z^n|^2  - |z^{\ast}|^2)^2  \dx\le \tfrac{4}{\beta} (E(z^n) -E^{\ast}) \overset{n\rightarrow\infty}\longrightarrow 0,
\end{align}
where the global convergence of the energies is ensured by Theorem \ref{lemma-weak-convergence-Euler}. Since $z^n,z^{\ast}\ge 0$, we conclude convergence of the whole sequence $z^{n}$ to $z^{\ast}$. That means that all strong $H^1$-limits of subsequences in Theorem \ref{lemma-weak-convergence-Euler} must coincide. Lemma \ref{lemma-uniqueness-positive-critical-points}, the uniqueness of the nonnegative eigenstates, finishes the proof.
\end{proof}
\begin{remark}
Elliptic regularity theory provides $H^2$- and $L^{\infty}$-bounds for $\mathcal{G}_{z}(v)$ which are of the form
\begin{align*}
\| \mathcal{G}_{z}(v) \|_{L^{\infty}(\Omega)} \lesssim \| \mathcal{G}_{z}(v) \|_{H^2(\Omega)} \lesssim
\| v \|_{L^2(\Omega)} \left(
1 + \| V\|_{L^{\infty}(\Omega)} + \beta \| z \|_{L^6(\Omega)}^2 \right).
\end{align*}
This implies that in the energy diminishing regime, the iterates $z^{n}$ remain pointwise uniformly bounded, with a bound that depends on $\beta$, $V$ and $E(z^0)$.
\end{remark}

\section{Numerical experiments}
\label{sec:num:experiments}

This section concerns the numerical performance of the proposed projected $a_z$-Sobolev gradient flow GF$a_z$ defined in \eqref{GFaz-eqn}. For a better assessment, we compare with established gradient flows, the GF$L^2$ iteration (or DNGF) from \eqref{GFL2-eqn}, the $H^1$-Sobolev gradient version GF$H^1$ from \eqref{GFH1-eqn} and the $a_0$-Sobolev gradient version GF$a_0$ from \eqref{GFa0-eqn} that incorporates the potential $V$. For the sake of a fair comparison of all methods, we use the (otherwise impractical) stopping criterion that the relative error with respect to some highly accurate (accuracy order $10^{-8}$) reference energy 
falls below the tolerance $\mbox{\rm\small TOL}=10^{-5}$. For the sake of simplicity, we measure performance in terms of number of iterations required to match this stopping criterion. This is a reasonable complexity indicator because the computational cost per iteration is essentially the same for all methods if a uniform step size $\tau$ is used. While GF$L^2$ and GF$a_z$ require the assembly of a new stiffness matrix from the previous density $|z^n|^2$ and one linear solve, GF$a_0$ and GF$H_1$ require two solves but the system matrices are invariant and do not need to be re-assembled. Our practical experience is that GF$L^2$ and GF$a_z$ iterations are slightly faster than the other two but this will not be taken into account in the following comparison.
 
As a general model, we solve the following  Gross-Pitaevskii eigenvalue problem: find $z^{\ast} \in H^1_0(\Omega)$ with $\| z^{\ast} \|_{L^2(\Omega)}=1$ and 
\begin{align}\label{weak-problem-new}
\frac{1}{2} (\nabla z^{\ast} , \nabla v )_{L^2(\Omega)} +  (V \hspace{2pt} z^{\ast} , v )_{L^2(\Omega)}  
+ \beta ( |z^{\ast}|^2 \hspace{2pt} z^{\ast} , v )_{L^2(\Omega)}  
 = \lambda^{\ast} ( z^{\ast} , v )_{L^2(\Omega)}
\end{align}
for all $v\in H^1_0(\Omega)$ and in a bounded domain $\Omega$ of $\R^2$. Note that the kinetic part, i.e. $(\nabla z^{\ast} , \nabla v )_{L^2(\Omega)}$, has an additional scaling factor $1/2$ compared to previously considered problem \eqref{weak-problem}. The particular choices of $\Omega$, $V$ and $\beta$ are specified separately in the various experiments. All problems are discretized using a $P1$-Lagrange finite element method on a uniform grid of width $h$ specified below.
Although adaptivity (as explained in Remark \ref{adaptive-GFaz}) can be used to improve the performance of GF$a_z$ (and also GF$H_1$, GF$a_0$), our comparisons focus on equidistant steps $\tau$. 

\begin{remark}
We stress that our comparison only aims at comparing the basic versions of the gradient flow methods and that each of these methods can be improved significantly with various techniques and hence the overall picture might change in this case. Here, we refer for example to the improvements of GF$L^2$ by using preconditioners and conjugated gradients as suggested in \cite{ALT17} or the improvements of GF$a_0$ by using Riemannian conjugate gradients as proposed in \cite{DaP17}. Such improvement can boost the performance dramatically compared to the basic versions of the gradient flow methods (cf. the numerical experiments in \cite{ALT17,DaP17}). Furthermore, adaptive mesh refinement strategies can improve the efficiency even further \cite{HSW19}. Another strategy, which can be particularly beneficial for excited states, is to use a different linearization technique that is based on the derivative of a scaling-invariant version of the Gross-Pitaevskii operator and which reacts more favorably to spectral shifts \cite{AHP19,JKM14}.
\end{remark}

\subsection{Model problem 1 - Ground states for a harmonic potential}
\label{subsec:mod1:groundstates}

In the first model problem, we consider \eqref{weak-problem-new} for a harmonic trapping potential with trapping frequencies $1/2$, i.e.
$$
V(x)=\tfrac{1}{2}|x|^2.
$$
The repulsion parameter $\beta$ is selected with three different values $\beta=10,100,1000$. Computing the corresponding Thomas-Fermi radii of the problem we restrict the computations to a square domain of the size $\Omega=(-6,6)^2$. The initial value $z^0$ is selected as the Thomas-Fermi density computed according to \cite{Bao14} using the exact ground state for $\beta = 0$. Since this is a nonnegative initial value, we expect all numerical approximations to converge to the unique positive ground state of $E$ (if $\tau$ is in the convergent regime). The ground state energies and eigenvalues for different values of $\beta$ are listed in Table \ref{table-energies}.

\begin{table}[h!]
\caption{\it Approximate ground state energies $E_{\mbox{\rm\tiny GS}}$ and corresponding ground state eigenvalues $\lambda_{\mbox{\rm\tiny GS}}$ for Model Problem 1 with $h=12\cdot 2^{-8}$ and different values for $\beta$.}
\label{table-energies}
\begin{center}
\begin{tabular}{|c||c|c|}
\hline $\beta$ & $E_{\mbox{\rm\tiny GS}}$  & $\lambda_{\mbox{\rm\tiny GS}}$ \\
\hline
\hline 10     &   0.79620688 & 2.06380 \\
\hline 100   &   1.97298868 & 5.75977 \\
\hline 1000 &   5.99303235   & 17.9771 \\
\hline
\end{tabular}\end{center}
\end{table}
Throughout our numerical experiments we observed that the stability regions for GF$H^1$ and GF$a_0$ are notably smaller than the ones for GF$L^2$ and GF$a_z$. Furthermore, the size of the spatial mesh size $h$ has essentially no influence on the convergence and number of steps required to fall below the tolerance. Both of these findings become visible in the results depicted in Table \ref{table-iterations-1} where we compare the different methods for the ad-hoc parameter choices $\tau=\tau_n=0.5$ and $\tau=\tau_n=1$ and for the mesh sizes $h=12\cdot 2^{-6}$ and  $h=12\cdot 2^{-8}$. With the default choice $\tau=1$, GF$L^2$ and GF$a_z$ perform equally well. In general we observe that GF$a_z$ is more sensitive with respect to the step size parameter $\tau$. 

\begin{table}[h!]
\caption{\it Model Problem 1: computation of ground states. 
The table shows the number of iterations obtained for the various methods for $\tau=0.5$ and $\tau=1$. The entry \quotes{$\infty$} means that the iteration did not converge. The spatial mesh size was selected as $h=12\cdot 2^{-6}$. The entries in brackets show the iteration count for higher spatial resolution $h=12\cdot 2^{-8}$. }
\label{table-iterations-1}
\begin{center}
\begin{tabular}{|c|c||c|c|c|c|c|c|}
\hline $\tau$ & $\beta$ & GF$L^2$  &  GF$H^1$ &  GF$a_0$ &  GF$a_{z}$ \\
\hline
\hline 1.0 & 10 & 9 (9)  & $\infty$ ($\infty$) & 7 (7) & 6 (7) \\
\hline 0.5 & 10 & 11 (11)& $\infty$ ($\infty$) & 14 (14)& 14 (14) \\
\hline
\hline 1.0 & 100 & 11 (12)& $\infty$ ($\infty$) & $\infty$ ($\infty$) & 9 (9)\\
\hline 0.5 & 100 & 13 (13)& $\infty$ ($\infty$) & $\infty$ ($\infty$) & 18 (18)\\
\hline
\hline 1.0 & 1000 & 15 (15)& $\infty$ ($\infty$) & $\infty$ ($\infty$) & 11 (11) \\
\hline 0.5 & 1000 & 16 (16)& $\infty$ ($\infty$) & $\infty$ ($\infty$) & 22 (22) \\
\hline
\end{tabular}\end{center}
\end{table}
Since the tables show only the results for two exemplary choices of $\tau$, it is more interesting to investigate what, for a fixed setup, is the minimum number of iterations that the methods require to reach the tolerance. We keep the step size $\tau$ constant. Corresponding results are depicted in Table \ref{table-iterations-3} for the three different values of $\beta$. We observe that the bigger $\beta$, the more iterations are required, though the growth is only moderate. We see that GF$a_{z}$ requires the fewest iterations, closely followed by GF$L^2$. Both  GF$H^1$ and GF$a_0$ perform decently, though they are considerably behind the other two approaches. We made the same observation in various experiments and assume that this is linked to the smaller stability domain of the GF$H^1$ and GF$a_0$, enforcing smaller values for $\tau$ and hence smaller updates in modulus. Optimal values for $\tau$ can be found by solving a minimization problem for $\tau$ in each time step (cf. \cite[Section 4]{DaK10}). If this is not done, the GF$H^1$ approach can be tough to use, because a stable constant time step is rather small. 

In Lemma \ref{sharp-lower-bounds-lemma} we observed the expected divergence (energy blow-up) for the GF$a_{z}$-approach for time steps $\tau\ge 2$. This bound seems to be pretty sharp according to further numerical experiments not presented here. All $H^1$ gradient flows share such a time step restriction. Only GF$L^2$ is unconditionally stable for all $\tau<\infty$.

\begin{table}[h!]
\caption{\it Model Problem 1: computation of ground states. The table shows the minimum number of iterations $N$ that the methods required to reach the error tolerance. Alongside $N$ we list one possible step size $\tau$ for which this number is reached. The spatial mesh size is fixed with $h=12\cdot 2^{-6}$.}
\label{table-iterations-3}
\begin{center}
\begin{tabular}{|c||c|c||c|c||c|c|}
\hline 
& $\tau$ & $N$ & $\tau$ & $N$ & $\tau$ & $N$ \\
\hline
\hline
& \multicolumn{2}{|c||}{$\beta = 10\hspace{10pt}$} 
& \multicolumn{2}{|c||}{$\beta = 100\hspace{5pt}$} 
& \multicolumn{2}{|c|}{$\beta = 1000$} 
 \\
\hline
GF$L^2$  & $5.0$ & $7$ %
& $5.0$ & $10$ %
& $5.0$ & $15$ \\ %
\hline
GF$H^1$  & $0.25$ & $41$ %
& $0.31$ & $25$ %
& $0.12$ & $48$ \\ %
\hline
GF$a_0$  & $1.0$ & $7$ %
& $0.34$ & $24$ %
& $0.1$ & $109$ \\%
\hline
GF$a_{z}$ & $1.2$ & $5$ %
& $1.1$ & $8$ %
& $1.1$ & $10$ \\ %
\hline
\end{tabular}\end{center}
\end{table}
\noindent The remaining experiments focus on a comparison between GF$a_{z}$ and GF$L^2$.

\subsection{Model problem 2 - Ground state in a lattice potential}
In the second model problem, again based on \eqref{weak-problem-new}, we investigate how the GF$a_z$ and GF$L^2$ methods perform when using a more complicated potential $V$ which consists of a harmonic part and an additional optical lattice. The potential is visualized in Figure \ref{potential_and_func} (left) and reads
\begin{align}
\label{pot-mod-prob-2}
V(x)=\frac{|x|^2}{2} + 20 + 20 \sin(2 \pi x_1) \sin(2 \pi x_2 ).
\end{align}
Furthermore, we use again $\Omega=(-6,6)^2$ and $\beta=1000$ and fix the mesh size $h=12\cdot 2^{-8}$. To compute the ground state of the corresponding energy functional, we start the different iterations with a Thomas-Fermi density that was computed according to  \cite{Bao14} (for the case of general potentials, which includes the lattice part in our case). The final ground state density is depicted in Figure \ref{potential_and_func} (right), where we identified the ground state energy with approximately $E_{\mbox{\rm\tiny GS}}=15.204825$ and the corresponding ground state eigenvalue with $\lambda_{\mbox{\rm\tiny GS}} = 36.708$.

\begin{figure}[h!]
\centering
\includegraphics[scale=0.29]{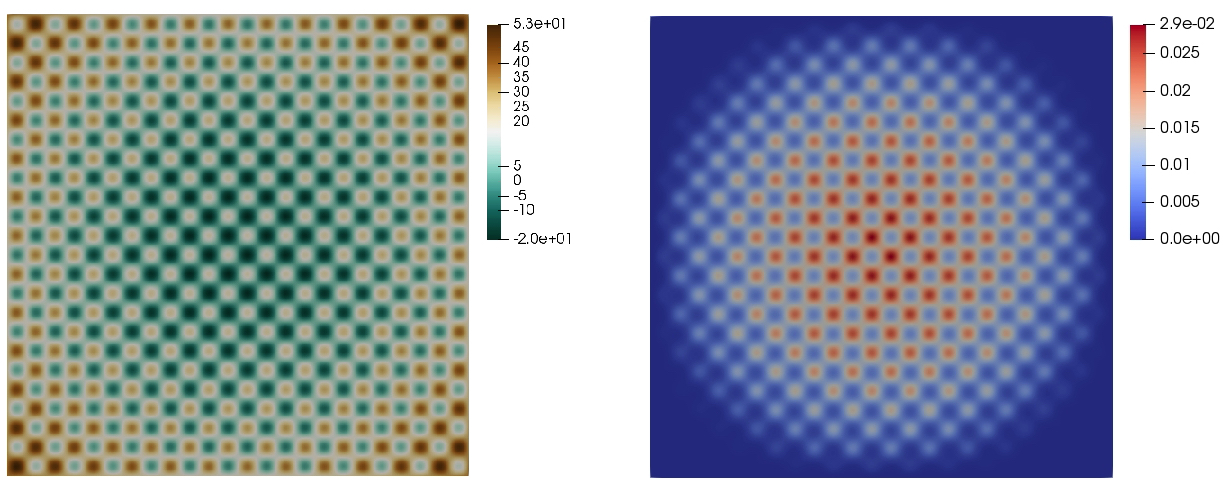}
\caption{\it Model Problem 2. Left: Visualization of the potential $V$ defined in \eqref{pot-mod-prob-2}. Right: Computed ground state density $|z^{\ast}|^2$.}
\label{potential_and_func}
\end{figure}

\begin{table}[h!]
\caption{\it Model problem 2. The table shows various step sizes $\tau$ and the corresponding number of GF$a_{z}$ and  GF$L^2$ iterations that are required to fall below the energy tolerance.}
\label{table-iterations-5}
\begin{center}
a)  GF$a_{z}\quad$ \begin{tabular}{|c||c|c|c|c|c|c|c|c|c|c|}
\hline 
$\tau$ & 0.8 & 0.9 & 1.0 & 1.1 & 1.2 & 1.3 & 1.4 & \bf 1.5 & \bf 1.6 &  1.7 \\
\hline
$N$ & 23 & 20 & 18 & 17 & 15 & 14 & 13 & \bf12 & \bf 12 & $\infty$ \\
\hline
\end{tabular}\\[1ex]
b) GF$L^2\quad$ \begin{tabular}{|c||c|c|c|c|c|c|c|c|c|c|c|}
	\hline 
	$\tau$ & 0.1 & 0.5 & 1 & 1.5 & 2 & 2.5 & 3 & \bf 5 & \bf 10 & \bf 100 & \bf 1000 \\
	\hline
	$N$ & 32 & 27 & 27 & 27 & 27 & 27 & 27 & \bf 26  & \bf 26 & \bf 26 & \bf 26 \\
	\hline
\end{tabular}
\end{center}
\end{table}

In Table \ref{table-iterations-5}(a) we see how the number of GF$a_{z}$ iterations vary depending on the selected step size $\tau$. The method is unstable for $\tau\leq 1.7$. For smaller time steps, the number of iterations decreases uniformly from 23 iterations for $\tau=0.8$ to $12$ iterations for $\tau=1.5$. Even though not contained in the table, the number of iterations for GF$a_{z}$ increases dramatically for $\tau\le 0.7$ and the method is no longer competitive in this regime. In practice we always recommend the usage of  adaptivity (cf. Remark \ref{adaptive-GFaz}) to find a good value for $\tau$. In this case only $11$ iterations were needed to achieve the error tolerance. The GF$L^2$ is less sensitive to the choice of time step. However, the minimal number of time steps $26$ is considerably higher to what is achieved by GF$a_{z}$. With the right choice of the step size, GF$a_{z}$ performs up to twice as fast. Using adaptivity the appropriate time step regime is easily reached.

Our general conclusion is that for simple test problems the GF$L^2$ and GF$a_{z}$ perform basically evenly. On the other hand, the GF$a_{z}$ can have visible advantages for more challenging test cases involving poor choices for the starting value $z^0$ or more complicated potentials.

\begin{remark}[Negative potentials, shift and invert]
Shifting the potential $V$ by $-20$ leads to a negative potential but does not affect the eigenfunctions. All energies and corresponding eigenvalues are simply shifted by $-20$ as well. The ground state energy level then reads $E_{\mbox{\rm\tiny GS}}=5.204825$ and the corresponding eigenvalue  $\lambda_{\mbox{\rm\tiny GS}} = 16.708$. Still, the negative potential causes problems for numerical simulation. We observed strong energy oscillations for the GF$L^2$ if the step size was not selected sufficiently small ($\tau<0.7$ in our tests). Such oscillations cannot happen if $V\ge 0$. Therefore it is reasonable to first shift $V$ so that it becomes positive, apply the methods to compute e.g. the ground state and afterwards shift the energy and the eigenvalue back to the original setup. This is equivalent to using a suitable shift parameter in a conventional inverse iteration method. 

As with linear eigenvalue problems, such a shift may as well be used to speed up convergence by increasing the relative sizes of spectral gaps. 
\end{remark}

\subsection{Model Problem 3 - Anderson Localization}
Our final numerical experiments is devoted to the phenomenon of {\it Anderson localization} \cite{And58}, which describes the exponential localization of waves in a disordered medium. In the context of the Gross-Pitaevskii eigenvalue problem this Anderson effect is reflected by strongly localized peaks in the ground state eigenfunction, provided that the potential $V$ is sufficiently disordered. 

\begin{table}[h!]
\caption{\it Model problem 3. The table shows various step sizes $\tau$ and the corresponding required number of GF$a_{z}$ and GF$L^2$ iterations to fall below the energy tolerance.}
\label{table-iterations-8}
\begin{center}
a)  GF$a_{z}\quad$  \begin{tabular}{|c||c|c|c|c|c|c|c|c|c|c|c|}
\hline 
$\tau$ 
& 1.0 & 1.1 & 1.2 & 1.3 & 1.4 & 1.5 & 1.6 & 1.7 & \bf 1.8 & \bf 1.9 & 2.0 \\ 
\hline
$N$ 
& 100 & 91 & 84 & 77 & 72 & 67 & 63 & 59 & \bf 56 & \bf 56 & $\infty$ \\
\hline
\end{tabular}\\[1ex]
b) GF$L^2\quad$ \begin{tabular}{|c||c|c|c|c|c|c|c|c|c|c|c|}
	\hline 
	$\tau$ 
	& 0.5 & 1 & 1.5 & 2 & 2.5 & 3 & 5 & 10 & \bf 100 & \bf 1000 \\
	\hline
	$N$ 
	& 80 & 76 & 74 & 74 & 73 & 73 & 72  & 72 & \bf 71 & \bf 71 \\
	\hline
\end{tabular}\end{center}
\end{table}
\begin{figure}[h!]
\centering
\includegraphics[scale=0.3]{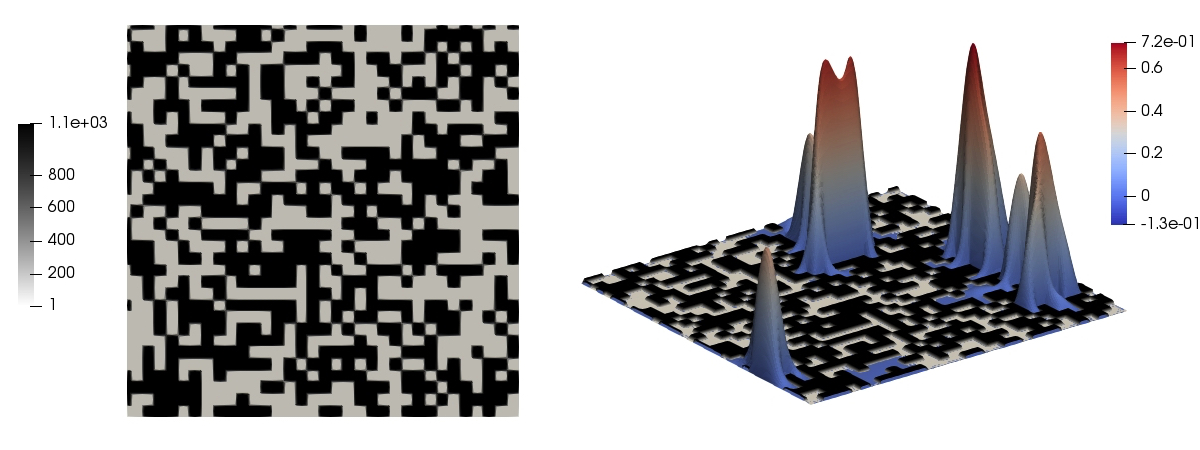}
\caption{\it Model Problem 3. Left image: Visualization of the random potential $V$, which varies between the values $1$ in the white regions and $\eps^{-2}\approx1000$ in the black regions. Right image: The Anderson-localized ground state $z_{\mbox{\rm\tiny GS}}$ consisting of several exponentially localized peaks.}
\label{anderson-localization}
\end{figure}

We consider \eqref{weak-problem-new} and let $\Omega=(-6,6)^2$ and $\beta=10$. The potential $V$ is a random disorder potential that is obtained by dividing $\Omega$ into $400 \times 400$ square cells with edge length $\eps=0.03$. In each cell independently, the potential takes either the value $V(x)=1$ or $V(x)=\eps^{-2}$ with equal probability. The scaling is selected according to the theoretical findings in \cite{AHP18}. The particular (deterministic) realization of $V$ used in our experiment is depicted in Figure \ref{anderson-localization}, together with the corresponding ground state $z_{\mbox{\rm\tiny GS}}$. We can clearly see the expected Anderson localization, as $z_{\mbox{\rm\tiny GS}}$ consists of few exponentially fast decaying peaks and is essentially zero elsewhere. With a highly accurate reference computation we obtained the ground state energy with $E_{\mbox{\rm\tiny GS}}=4.84223025$ and the ground state eigenvalue with $\lambda_{\mbox{\rm\tiny GS}}=10.826242$. The uniform mesh in our computations has the mesh size $h=12\cdot 2^{-8}$ which is fine enough to resolve the variations of the potential. The initial value was again selected as a suitable Thomas-Fermi approximation.

In Table \ref{table-iterations-8} we can see the number of iterations for GF$a_{z}$  and 
 GF$L^2$.  Again, we observe a similar performance of both methods, where GF$a_{z}$ shows stronger variations in the number of iterations. However, comparing the peak performance of the approaches, we see that GF$L^2$ is around 27\% slower than GF$a_{z}$. It is interesting to note that we observed convergence of GF$a_{z}$ until very close to the theoretical upper limit of $\tau=2$. Combining GF$a_{z}$ with an adaptive step size control as described at the beginning of this section, the number of iterations dropped even further from $56$ to $52$. In general we can conclude that both GF$L^2$ and GF$a_{z}$ are well-suited for an efficient computation of Anderson localized ground states, where the GF$a_{z}$ with adaptivity shows clearly the best performance.

\medskip
$\\$
{\bf Acknowledgements.}
The authors thank Robert Altmann for the fruitful discussions and valuable comments on some of the proofs. Furthermore, we thank the anonymous reviewers for their very insightful comments that greatly improved the contents of this paper.

\def\cprime{$'$}

\end{document}